\documentclass[12pt]{amsart}
\numberwithin{equation}{section}

  \usepackage{verbatim}
\usepackage{mathrsfs}
\usepackage{latexsym}
\usepackage{epsfig}
\usepackage{amsmath}
\usepackage{bbm}
\usepackage{graphics}
\usepackage{stmaryrd}
\usepackage{amssymb}
\usepackage{amsthm}
\usepackage{amsrefs}
\usepackage{amsopn}
\usepackage{amscd}
\usepackage[all,knot]{xy}
\xyoption{all}
\usepackage{rotating}
\usepackage{hyperref}

\theoremstyle{plain}
\newtheorem{thm}{Theorem}[section]
\newtheorem{lem}[thm]{Lemma}
\newtheorem{prop}[thm]{Proposition}
\newtheorem{cor}[thm]{Corollary}

\newtheorem*{thm*}{Theorem}
\newtheorem*{lem*}{Lemma}
\newtheorem*{prop*}{Proposition}
\newtheorem*{cor*}{Corollary}

\theoremstyle{definition}
\newtheorem{defn}[thm]{Definition}
\newtheorem*{defn*}{Definition}

\newtheorem{ex}[thm]{Example}
{}
\newtheorem{rem}[thm]{Remark}
\newtheorem*{rem*}{Remark}

{}
{}
{}
\newtheorem*{ack}{Acknowledgements}{}

\theoremstyle{remark}
{}
{}
{}

\def\tensor{\otimes}

\def\to{\longrightarrow}

\def\ZZ{\mathbb{Z}}

\def\cF{\mathcal{F}}

\def\S{\Sigma}

\def\sfD{\mathsf{D}}
\def\sfK{\mathsf{K}}
\def\sfT{\mathsf{T}}

\def\mcD{\mathcal{D}}

\def\mcE{\mathcal{E}}
\def\mcF{\mathcal{F}}

\def\op{\mathrm{op}}

\DeclareMathOperator{\Proj}{Proj}

\DeclareMathOperator{\Hom}{Hom}

\DeclareMathOperator{\Ext}{Ext}

\DeclareMathOperator{\Modu}{Mod}

\DeclareMathOperator{\thick}{\mathrm{thick}}

\DeclareMathOperator{\Loc}{\mathrm{Loc}}

\DeclareMathOperator{\Thick}{\mathrm{thick}}
\DeclareMathOperator{\rad}{\mathrm{rad}}

\input{xypic}
\newcommand{\hash}{\#}
\newcommand{\lra}{\longrightarrow}
\newcommand{\lla}{\longleftarrow}
\newcommand{\finbuilds}{\models}
\newcommand{\finbuiltby}{=\!\!\!|}
\newcommand{\builds}{\vdash}
\newcommand{\DRfbbS}{\sfD(R \finbuiltby_{q^*} S)}
\newcommand{\DRfbbR}{\sfD(R \finbuiltby  \; R)}
\newcommand{\DRfbbk}{\sfD(R \finbuiltby \; k)}

\newcommand{\DEfbbE}{\sfD(\mcE \finbuiltby \; \mcE)}
\newcommand{\DEfbbk}{\sfD(\mcE \finbuiltby \;k)}
\newcommand{\DEfbbD}{\sfD(\mcE \finbuiltby_{i^*}\mcD)}

\newcommand{\Qh}{\hat{Q}}
\newcommand{\Rh}{\hat{R}}
\newcommand{\Sh}{\hat{S}}
\newcommand{\Qmod}{\mbox{Mod-$Q$}}
\newcommand{\Qhmod}{\mbox{Mod-$\Qh$}}
\newcommand{\Rmod}{\mbox{Mod-$R$}}
\newcommand{\AMod}{\mbox{Mod-$A$}}
\newcommand{\Rhmod}{\mbox{Mod-$\Rh$}}
\newcommand{\Smod}{\mbox{Mod-$S$}}
\newcommand{\Shmod}{\mbox{Mod-$\Sh$}}

\newcommand{\modD}{\mbox{Mod-$\mcD$}}
\newcommand{\modE}{\mbox{Mod-$\cE$}}
\newcommand{\modF}{\mbox{Mod-$\cF$}}
\newcommand{\kdq}{k^{\hash Q}}
\newcommand{\kdr}{k^{\hash R}}
\newcommand{\kds}{k^{\hash S}}

\newcommand{\kde}{k^{\hash \mcE}}

\newcommand{\qh}{\hat{q}}
\newcommand{\ph}{\hat{p}}
\newcommand{\cE}{\mathcal{E}}
\newcommand{\HomQ}{\Hom_Q}
\newcommand{\HomR}{\Hom_R}
\newcommand{\HomS}{\Hom_S}
\newcommand{\HomD}{\Hom_{\mcD}}
\newcommand{\HomE}{\Hom_{\mcE}}
\newcommand{\HomF}{\Hom_{\mcF}}
\newcommand{\Eb}{\overline{E}}
\newcommand{\Fb}{\overline{F}}
\newcommand{\Db}{\overline{D}}
\newcommand{\st}{ \; | \; }

\newcommand{\T}{\mathbb{T}}
\newcommand{\sdr}{\rtimes}
\newcommand{\Dsg}{\sfD_{\mathrm{sg}}}
\newcommand{\Dcosg}{\sfD_{\mathrm{cosg}}}
\newcommand{\Q}{  {\mathbb{Q}}  }
\newcommand{\fm}{\mathfrak{m}}
\newcommand{\Fp}{\mathbb{F}_p}

\title{Morita theory and singularity categories}

\author{J.P.C.Greenlees}
\address{J.P.C.Greenlees, Warwick Mathematics and Institute, Zeeman Building, 
Coventry, CV4 7AL. UK.}
\email{john.greenlees@warwick.ac.uk}

\author{Greg Stevenson}
\address{Greg Stevenson, School of Mathematics and Statistics,
University of Glasgow,
University Place,
Glasgow G12 8QQ
}
\email{gregory.stevenson@glasgow.ac.uk}

\thanks{The authors began the discussions that led to this paper
  during the IRTATCA Programme at the  CRM Barcelona in early 2015. We are very
  grateful to the organizers and the CRM for bringing together such a stimulating group of
  people, and providing an excellent environment. The first author is
  grateful to the Simons Foundation for support.}

\subjclass[2010]{13D09, 16E45, 20J06, 55P43, 55P62}

\begin{document}

\begin{abstract}
\noindent 
We propose an analogue of the bounded derived category for an
augmented ring spectrum, defined in terms of a notion of Noether
normalization. In many cases we show  this category is independent of the chosen
normalization. Based on this, we define the singularity and
cosingularity categories measuring the failure of regularity and
coregularity and prove  they are Koszul dual in the style of the BGG
correspondence.  Examples of interest include Koszul algebras and Ginzburg DG-algebras, 
$C^*(BG)$ for finite groups (or for compact Lie groups with orientable
adjoint representation), cochains in rational homotopy theory and 
various examples from chromatic homotopy theory.   
\end{abstract}

\maketitle

\tableofcontents

\section{Introduction}
\subsection{Aspiration}
The singularity category of a commutative Noetherian ring $R$ is the Verdier quotient 
$$\Dsg (R)=\frac{\sfD^\mathrm{b}(R)}{\sfD^\mathrm{c}(R)}$$
of the bounded derived category, which consists of complexes with finitely generated total cohomology, by the 
 bounded complexes of
finitely generated projectives. When $R$ is regular, every finitely
generated module has a finite resolution by finitely generated
projectives, so that $\Dsg (R)=0$. The converse is also true, and thus $\Dsg (R)$ measures the deviation from regularity.

One would like to have such a measure of `regularity' for rings in other contexts. The ones
we have in mind are differential graded algebras (DGAs), for instance
those coming from rational homotopy theory, and ring spectra, for
example the ring spectra $C^*(BG; k)$ coming from modular representation theory. Accordingly, 
our central motivation is to generalize the definition of singularity
category by replacing $R$ with a DGA or a ring spectrum. The fundamental difficulty is that of giving
 good notions of `finitely generated' and `bounded'. 

The test of our success is in the examples we are able to cover: these
include Koszul algebras and Ginzburg DG-algebras, 
$C^*(BG)$ for finite groups, cochains in rational homotopy theory and 
various examples from chromatic homotopy theory.   

\subsection{The bounded derived category}
Although our motivation was indeed through the singularity category,
 experience teaches us that the bounded derived category of finitely
 generated modules is more
fundamental. 

In particular,  $\sfD^\mathrm{b}(R)$  often has
better properties than $\sfD^\mathrm{c}(R)$. For instance, if $R$ is a
$k$-algebra essentially of finite type for some field $k$ then (an enhancement of)
$\sfD^\mathrm{c}(R)$ is homologically smooth over $k$  (i.e.\ the diagonal is a small $\sfD^\mathrm{c}(R)$-bimodule) if and only if $R$ is smooth. On the other hand, (an enhancement of) $\sfD^\mathrm{b}(R)$ is frequently homologically smooth even when $R$ is singular (see \cite{Lunts}*{Theorem~6.3}). In a similar vein, $\sfD^\mathrm{b}(R)$ is known to be strongly generated in many cases while $\sfD^\mathrm{c}(R)$ can only be strongly generated if $R$ is regular.

 It turns out to be very effective to use this `derived smoothness' or `regularity' even for
 singular $R$.  Homological smoothness localises: notwithstanding the
 terminology, singularity categories are generally smooth. It is
 helpful to view this smoothness  as a categorical completeness
 condition; from this point of view one obtains $\sfD^\mathrm{b}(R)$
 by closing $\sfD^\mathrm{c}(R)$ under certain homotopy colimits (cf.\ \cite{neemanTBC}*{Theorem~0.14} and the preceding discussion). Explicitly, the projective resolution of a finitely
 generated module of infinite projective dimension can be viewed as
 the colimit of its brutal truncations, all of which are bounded
 complexes of finitely generated projectives and hence small. One
 useful consequence of this completeness is an analogue of Brown representability which holds for strongly generated triangulated categories and which is exploited in Section~\ref{sec:dependence}. 

The bounded derived category also naturally arises in many contexts
such as Grothendieck duality and Koszul duality; being somewhat larger
than $\sfD^\mathrm{c}(R)$ in the singular case often makes it a less
rigid object. 

In view of the importance of the bounded derived category, the fact
that we extend its definition to wider contexts is an important
secondary benefit.

\subsection{The definition}
For the purposes of the introduction,  we imagine beginning with a ring spectrum $R$ and a
map $R\lra k$ to a field $k$. We will recall relevant background in
Section \ref{sec:rnfg}, but  readers wishing to think concretely may
consider an ordinary local ring with residue field $k$  or $R=C^*(BG)\lra k$. Numerous other
examples are provided in Section~\ref{sec:rnfg}. 
The definition is based upon a choice of ``Noether normalization'' i.e.\ a morphism $S\stackrel{q}{\lra} R$ such that both $R$ and $k$ are small over $S$. Then, inspired by commutative algebra, one defines a bounded derived category relative to this normalization
\begin{displaymath}
\sfD^{q\mathrm{-b}}(R)=\{ M\in \sfD(R) \st M\; \text{is small when restricted to}\; S\}.
\end{displaymath}
By construction this contains both $R$ and $k$, and so, being thick, contains
$\sfD^\mathrm{c}(R)$ and the objects with finite dimensional
homotopy. In particular, it allows us to define the singularity category
as $\sfD_{q-\mathrm{sg}}(R)=\sfD^{q\mathrm{-b}}(R)/\sfD^c(R)$
and the cosingularity category
as $\sfD_{q-\mathrm{cosg}}(R)=\sfD^{q\mathrm{-b}}(R)/\Thick_R(k)$, which measures how far $R$ is from having finite dimensional homotopy.
%propose putting some stuff on how define D^b here, then lead with line about challenge, then BGG example and say use as template and give general result. Then quick sketch of contents and where stuff is.

\subsection{Proving the definition}
In principle we can justify the definition by showing it is useful,
but we will in fact show that this notion of finite generation is 
 intrinsic in the sense that  it does not depend on
the choice of normalization.  Our most effective result is Corollary
\ref{cor:Gorfg}: if $R$ is complete then any two relatively Gorenstein
normalizations define the same notion of finite generation and give
the same bounded derived category. 

This is very striking for  $R=C^*(BG)$. It states  all normalizations
of $C^*(BG)$ by a ring of the same type
give the same notion: a module is finitely generated if and only if
its {\em cohomology} is finitely generated over $H^*(BG)$. In particular, 
if $G$ is a $p$-group any $C^*(BG)$-module with finitely generated cohomology 
is small (see Corollary~\ref{cor:c-small} and Example~\ref{ex:BG-c-small}). 

In Section~\ref{sec:dependence} we give an approach using
representation theoretic methods: the highlights are   
Proposition~\ref{prop:cfgisf} and Corollary \ref{cor:sgfgislfp}. The former gives  
a direct interpretation of $\sfD^{q\mathrm{-b}}(R)$ in terms of finite  
generation of homotopy groups when the homotopy of $S$ is itself  
regular. The latter relates $\sfD^{q\mathrm{-b}}(R)$ to another intrinsically  
defined finiteness condition, phrased in terms of presheaves on  
$\sfD^\mathrm{c}(R)$, which characterises finite generation with respect to  
smooth normalizations with coherent homotopy.  
 
\subsection{Koszul duality and the BGG correspondence}
 The basis  of our attempts to understand $\sfD^{q\mathrm{-b}}(R)$ and
its singularity  and cosingularity quotients is the theory of Koszul duality.

The classic in this genre is the BGG
correspondence which relates the singularity category of the standard graded exterior
algebra $\Lambda (\tau_0, \ldots , \tau_n)$ to a well known invariant of its Koszul dual polynomial ring
 $k[x_0, \ldots , x_n]$: 
$$\Dsg (\Lambda (\tau_0, \ldots , \tau_n))=
\frac
{\sfD^\mathrm{b}(\Lambda (\tau_0, \ldots , \tau_n))}
{\sfD^\mathrm{c}(\Lambda (\tau_0, \ldots , \tau_n))}
\simeq 
\frac
{\sfD^\mathrm{b}( k[x_0, \ldots , x_n])}
{\sfD^\mathrm{b}_{tors}  (k[x_0, \ldots , x_n])} 
\simeq 
\sfD^\mathrm{b}(\mathbb{P}^n_k),  $$
where $\sfD^\mathrm{b}_{tors}  (k[x_0, \ldots , x_n])$ consists of complexes whose
homology is finite dimensional as a vector space. 

We prove an analogue for sufficiently well-behaved normalizations $S\stackrel{q}{\lra} R$. In fact, the above story is a consequence of an equivalence at the level of bounded derived categories
\begin{displaymath}
\sfD^\mathrm{b}(\Lambda (\tau_0, \ldots , \tau_n)) \simeq \sfD^\mathrm{b}( k[x_0, \ldots , x_n]),
\end{displaymath}
which interchanges the  bounded complexes of finitely generated
projectives and the complexes with finite dimensional cohomology. We
give a substantial generalization of 
this equivalence. In Section~\ref{sec:standardcontext} we introduce
the Koszul dual of the cofibre sequence arising from a
normalization. Under favourable circumstances, given a normalization
$S\lra R$ with cofibre $Q=R\tensor_Sk$,  one may take derived endomorphisms of
$k$, to obtain a dual cofibre sequence
\begin{displaymath}
\cF = \Hom_S(k,k) \lla \cE = \Hom_R(k,k) \lla \mcD = \Hom_Q(k,k)
\end{displaymath}
where the morphism $\mcD \lra \cE$ is a normalization in the same
sense. A number of nice properties that these cofibre sequences may
 have are formalized in Section~\ref{sec:SGC} by the notion of a
 Symmetric Gorenstein Context. Roughly, it says that all of the six
 rings and four morphisms occuring in the two cofibre sequences are 
Gorenstein and both sequences arise from taking the cofibre of a
normalization.  We show that under completeness hypotheses all these
good properties follow from the requirements on the original
normalization  $S\lra R$.

Our main theorem is as follows. 
\begin{thm*}[\ref{thm_realmain}, \ref{thm:main}]
Suppose $S\stackrel{q}{\lra}R$ is such that $R$ and $S$ are complete, both $R$ and $k$ are small over $S$, and we have
\begin{displaymath}
\Hom_S(R,S) \simeq \Sigma^{a_q}R \quad \text{and} \quad \Hom_S(k,S) \simeq \Sigma^{a_S}k
\end{displaymath}
for some $a_q, a_S\in \ZZ$. Then
\begin{displaymath}
\cE = \Hom_R(k,k) \stackrel{i}{\lla} \mcD = \Hom_Q(k,k)
\end{displaymath}
is a normalization and if in addition $\cE$ satisfies $\Hom_\cE(k,\cE)
\simeq \Sigma^{a_\cE}k$ for some integer $a_\cE$ (as is automatic if $R$ is an augmented 
$k$-algebra), there is an equivalence
\begin{displaymath}
\sfD^{q\mathrm{-b}}(R) \simeq \sfD^{i\mathrm{-b}}(\cE)
\end{displaymath}
interchanging the small objects with $\thick(k)$. In particular, there are equivalences
\begin{displaymath}
\sfD_{q-\mathrm{sg}}(R) = \frac{\sfD^{q\mathrm{-b}}(R)}{\sfD^\mathrm{c}(R)} \simeq \frac{\sfD^{i\mathrm{-b}}(\cE)}{\thick_{\sfD(\cE)}(k)}
= \sfD_{i-\mathrm{cosg}}(\mcE)
\end{displaymath}
and
\begin{displaymath}
\sfD_{q-\mathrm{cosg}}(R) = \frac{\sfD^{q\mathrm{-b}}(R)}{\thick_{\sfD(R)}(k)}\simeq \frac{\sfD^{i\mathrm{-b}}(\cE)}{\sfD^\mathrm{c}(\cE)}
= \sfD_{i-\mathrm{sg}}(\mcE).
\end{displaymath}
\end{thm*}

\subsection{Examples}
In Section~\ref{sec:egs2} we conclude by giving a number of concrete examples to illustrate the theorems. To give just a hint
of these: they  range from
standard examples of Koszul duality in algebra (Examples
\ref{eg:DsgKoszul}, \ref{eg:DsgGinzburg}) giving a new point of view
on some known equivalences,  through rational
homotopy theory (Example \ref{ex:rat11}):
$$\sfD_\mathrm{sg}(C^*(X))\simeq \sfD_\mathrm{cosg}(C_*(\Omega X)), $$
 to  ring spectra arising from modular representation
theory (Examples \ref{ex:stmodkG} to \ref{eg:DsgA4}) and chromatic
homotopy theory (Example \ref{eg:Dsgchrom}). 
Two notable counterparts of the BGG correspondence above are the
equivalence (Example \ref{ex:stmodkG}):
$$\sfD_\mathrm{cosg}(C^*(BG))\simeq \mathrm{stmod}(kG)$$
for $p$-groups $G$ relating modules over $C^*(BG)$ to the stable
module category, and some counterparts in chromatic homotopy theory 
(Example \ref{eg:Dsgchrom}), which we illustrate here with  connective
real $K$-theory and its connection with the subalgebra $\mathcal{A}(1)$ of the
Steenrod algebra: 
$$\sfD_\mathrm{cosg}(ko)\simeq \mathrm{stmod}([\mathcal{A}(1)]),  $$
where $[\mathcal{A}(1)]$ is a ring spectrum with homotopy
$\mathcal{A}(1)$.  

We recommend the reader glances through Section \ref{sec:egs2} to
understand why we make an effort to keep the context very general. 

\subsection{Contents}
We begin in Section  \ref{sec:sundries} by introducing some standard
notation and terminology. 

In Section \ref{sec:rnfg} we give our main definitions: the notion of normalization and
the resulting definition of `finitely generated', and the bounded
derived category. We introduce several examples and describe briefly how this applies. We also comment on the generality in which normalizations exist.

In Section~\ref{sec:dependence} we give a first study of the dependence of $\sfD^{q\mathrm{-b}}(R)$ on 
the choice of normalization $S\stackrel{q}{\lra}R$, using techniques
from representation theory. 

In Section \ref{sec:standardcontext} we describe how a normalization
gives rise to the Six Ring Context consisting of two Koszul dual
cofibre sequences. In Section \ref{sec:SGC} we restrict attention to 
Symmetric Gorenstein Contexts where  all the rings and maps are
Gorenstein and the two cofibre sequences are dual. We show that in the 
complete context,  the conditions on the original normalization alone
are often sufficient to ensure we have the full Symmetric Gorenstein 
Context. We show that this often happens in our examples. 

In Section \ref{sec:completion} we recall the appropriate derived
notions of completion, and show that in the complete case all
Gorenstein normalizations give the same notion of finite generation
and the same bounded derived category. 

In Section \ref{sec:cr} we show that in 
the Standard Gorenstein Context,  the Morita equivalences, change
of rings and completions are well related, giving eight valuable commutation
relations: four direct and four with dimension shifts. Finally,
having established the formal framework,   it is straightforward to 
prove our main theorem in  Section \ref{sec:duality}. We illustrate the result in
our examples in Section \ref{sec:egs2}.

We finish, in Section~\ref{sec:glossary}, with a glossary; our constructions involve making a number of definitions and the terminology is collected there, together with references to where it appears in the article.

\begin{ack}
We thank A.J.Baker for simplifying our argument in Example
\ref{eg:Dsgchrom}. We are indebted  to the referee for their very careful reading of this
manuscript; their comments greatly improved the end result. 
\end{ack}

\section{Sundries}\label{sec:sundries}

In this section we fix various notation and conventions that will be used throughout the sequel. In particular, due to the range of examples we treat there are, somewhat inevitably, challenges involving the terminology which we address before continuing. An extensive list of terminology can be found in Section~\ref{sec:glossary}.

We will use the term `ring' to mean structured ring spectrum and note
that this encompasses the theory of DG-algebras (see \cite{ShipleyHZ}
for details). Along these lines, given a DG-algebra $A$, for instance
a (classical, discrete) ring, we will tacitly identify $A$ with its Eilenberg-Mac Lane spectrum $HA$. By \cite{SchwedeShipleymodule}*{Theorem~5.1.6} we have $\sfD(A)\simeq \sfD(HA)$ so this does no harm. To illustrate this, let us mention that throughout we will generally work over a field $k$ by which we really mean its Eilenberg-Mac Lane spectrum $Hk$.

Given a spectrum $X$ we will denote its homotopy groups $\pi_*X$ by
$X_*$. For instance, the coefficient ring of a ring spectrum $R$ will
be denoted $R_*$. If the ring $R$ were $HA$ for some DG-algebra $A$
this would be the same as $H_*(A)$, the homology of $A$. We will
choose between homological and homotopical language depending on
the context; many of our examples will be rings of the form
$C^*(X;k)$, for some space $X$, and accordingly  $\pi_*C^*(X;k)$ is
the cohomology of $X$, i.e.\ $H^*(X;k)$ (with upper and lower gradings
related by  $M^k=M_{-k}$ as usual). 

Now let us fix a ring $R$ and introduce some of the associated
notation. By $\Rmod$ we mean the model category (or stable
$\infty$-category) of $R$-module spectra with weak equivalences the
 maps inducing weak equivalences of the underlying spectra. The homotopy category of $\Rmod$ is
$\sfD(R)$ the derived category of $R$. Given an object $X$ of
$\sfD(R)$ we denote by $\Thick(X)$, or $\Thick_R(X)$ if the ring needs
to be emphasised, the smallest full replete subcategory of $\sfD(R)$
containing $X$ and closed under suspensions, mapping cones, and
retracts and call it the thick subcategory generated by $X$. We denote
by $\Loc(X)$ the localizing subcategory generated by $X$ which is the
smallest full replete subcategory containing $X$ and closed under arbitrary
coproducts, suspensions, and mapping cones.

If $Y\in \Thick(X)$ we will say $X$ {\em finitely builds} $Y$ and
write $X\finbuilds Y$, and if $Y\in \Loc(X)$ we say $X$ {\em builds}
$Y$ and write $X\builds Y$. The thick subcategory of small (more
precisely,  $\aleph_0$-small) objects of $\sfD(R)$ is
\begin{displaymath}
\sfD^\mathrm{c}(R) := \{X\in \sfD(R) \; \vert \; R\finbuilds X\}
\end{displaymath}
and can also be characterised as consisting of those objects such that the corresponding corepresentable functor preserves arbitrary coproducts. It is necessary at this point to say something about the terminology: there are many synonyms for small. In algebraic settings it is customary to call objects of $\sfD^\mathrm{c}(R)$ perfect and in abstract settings to call them compact. The latter is reflected in the notation, which is by this point quite standard so we stick with it. However, we will consistently use the descriptors small or finitely built by $R$ rather than perfect or compact. We will also be concerned with a number of other subcategories of $\sfD(R)$ which are defined throughout the article.

All functors throughout are derived and so we do not indicate this in
the notation. For instance, given $R$-module spectra $X$ and $Y$ we
denote by $\Hom_R(X,Y)$ the (derived) mapping spectrum. In a similar vein all tensor products are derived, by cofibre we mean homotopy cofibre (in the ambient category\textemdash{}for instance, here it refers to the homotopy cofibre of a map of module spectra), and so on.

Given a map of rings $S\stackrel{q}{\lra}R$ we denote base change and
restriction by $q_*$ and $q^*$ respectively. To be completely clear,
since we cover many contexts our notation reflects the variance of
the functors and not that of the functors on  the associated geometric
objects: throughout we have
\begin{displaymath}
q_* = R\otimes_S - \quad \text{and} \quad q^* = \Hom_R({}_SR,-),
\end{displaymath}
where, as noted above, everything is tacitly derived.

\section{Regularity, normalization and finite generation}
\label{sec:rnfg}
We are working in the context of homotopy invariant commutative-inspired
algebra. We collect here some of the basic definitions, and provide
pointers to the literature. We then introduce the concept of a normalization which is at the heart of all that follows. Throughout $R$ is some ring spectrum; several examples of a suitable choice of $R$ will be provided throughout.

\subsection{Regularity}\label{sec:regularity}
We say that  $R\lra k$ is {\em g-regular} if $k$ is small as an
$R$-module, i.e.\ $R$ finitely builds $k$. By the Auslander-Buchsbaum-Serre theorem a commutative Noetherian local ring with residue
field $k$ is g-regular if and only if it is regular. 
%Accordingly we drop the letter g from now on. 
We will say that $S\lra R$ is {\em
  relatively g-regular} if $R$ is small as an $S$-module. 

%Dually we say that $R\lra k$ is {\em coregular} if $k$ finitely builds
%$R$, and $S\lra R$ is {\em relatively coregular} if $R$ finitely
%builds $S$.

\subsection{Proxy-regularity}
Since g-regularity is an extremely strong condition we use the following
much weaker condition as a basic finiteness condition. 

\begin{defn}\cite{DGI2}
\label{defn:proxysmall}
We say that $k$ is {\em proxy-small} if there is an object 
$K$ with the following properties
\begin{itemize}
\item $K$ is small ($R\finbuilds K$),
\item  $K$ is finitely built from $k$ ($k\finbuilds K$) and
\item  $k$ is built from $K$ ($K\builds k$). 
\end{itemize}
In this case we say that $R$ is \emph{proxy-regular}.
\end{defn}

One of the main messages of \cite{DGI2} is that we might use the
condition that $k$ is proxy-small as a substitute for the Noetherian
condition in the conventional setting. This rather weak condition 
allows one to develop a very useful theory applicable in a large range
of examples. 

We can illustrate this by looking at the proxy-small condition in the classical case. 

\begin{ex}   {\em (Algebra)}
When $R$ is a commutative Noetherian
local ring with residue field $k$, the Auslander-Buchsbaum-Serre theorem states that $k$ is  small if and only if $R$ is a regular 
local ring. This confirms that  the smallness of $k$ is a very 
strong condition. On the other hand, $k$ is  always proxy-small:  
we may take $K$ to be the Koszul complex for a generating 
sequence for the maximal ideal.
\end{ex}

We now consider the situation in a number of more
complicated contexts; we take this as an opportunity to set up conventions and notation for examples that we will refer to throughout, which give life and form to the abstraction that follows.

\begin{ex}\label{ex:dt}   {\em (Rational homotopy theory)}
We may take $R$ to be a commutative DGA over the rationals. For
example, if we insist $R$ is coconnective and simply coconnected, the
category of these is equivalent to that of rational spaces \cite{Quillen}. 
We therefore take $R=C^*(X;\Q)$ and $k=\Q$.

We see that $R$ is regular if and only if $X$ is a finite product of even
Eilenberg-MacLane spaces $K(\Q, 2n)$. Indeed, since $X$ is 1-connected the Eilenberg-Moore theorem states
$$\cE =\Hom_{C^*(X)}(k, k)\simeq C_*(\Omega X; \Q). $$
We then note that $\Omega X\simeq
\prod_nK(\pi_nX, n-1)$, which has finite homology if and only if the
product is finite and the Eilenberg-MacLane spaces  are all in odd degree.

On the other hand, $\Q$ is proxy-small whenever $H^*(X)$ is
Noetherian. Taking a usual Noether normalization we see $H^*(X)$ is finite as a module over a polynomial
subring. We may then realize this polynomial subring by a map $X\lra
\prod_iK(\Q, 2n_i)$, with fibre $F$, and we will denote by $S$ the ring $C^*(\prod_iK(\Q, 2n_i);\Q)$. We may take $K=C^*(F; \Q)$ as a proxy for $\Q$; this
builds $\Q$ since $K$ is a ring, $R\simeq R\tensor_SS \finbuilds R\tensor_S\Q\simeq C^*(F;\Q)=K$, and 
$\Q\finbuilds C^*(F;\Q)$ because $H^*(X; \Q)$ is finite over the
polynomial subring. 
\end{ex}

\begin{comment}
\begin{ex}   {\em ($p$-complete spaces)}
Returning to Example \ref{ex:t}, with $R=C^*(X;\Fp)$, we see that $R$ is
regular if and only if $H_*(\Omega X; \Fp)$ is finite
dimensional (cf.\ Dwyer-Wilkerson $p$-compact
groups \cite{DwyerWilkerson}). We restrict comments on proxy smallness to the
representation theoretic case.  
\end{ex}
\end{comment}

\begin{ex}\label{ex:g}   {\em (Representation theory)} 
We could consider a compact Lie group $G$, set $k=\Fp$, and take $R=C^*(BG;\Fp)$. This example satisfies the hypotheses of the Eilenberg-Moore theorem so that
$$\cE =\Hom_{C^*(BG)}(k, k)\simeq C_*(\Omega(BG_p^{\wedge}); \Fp), $$
where $BG_p^{\wedge}$ denotes the Bousfield-Kan $p$-completion of
$BG$.

If $G$ is a finite  $p$-group, $BG$ is already $p$-complete, so that
$\Omega (BG_p^{\wedge})\simeq G$, and again if $G$ is  connected
$\Omega (BG_p^{\wedge})\simeq G_p^{\wedge}$,  but in general $\Omega
(BG_p^{\wedge})$ will be infinite dimensional.
 
In this case $R$ is regular if and only if $H_*(\Omega
(BG_p^{\wedge}); \Fp)$ is finite dimensional
(i.e., $\Omega (BG_p^{\wedge})$ is a $p$-compact group in the sense of
Dwyer-Wilkerson).  We have already observed that this happens if $G$ is a finite
$p$-group or a connected compact Lie group. 

It is shown in \cite[Subsection 5.7]{DGI2} that $C^*(BG)$ is proxy-regular (i.e.\ $k$
is proxy-small) for all compact Lie groups $G$. 
\end{ex}

\subsection{Normalization and finitely generated
  modules}\label{sec:Db}
We need a well behaved notion of finite generation for $R$-modules
$M$.  The most naive notion is finite generation of the coefficients:

\begin{defn}\label{defn:c-fg}
We say an $R$-module $M$ is {\em coefficient-finitely generated} if the module $M_*$
of homotopy groups is finitely generated over the coefficient ring
$R_*$. There is a naturally corresponding subcategory
$$\sfD^f(R)=\{ M\st M_* \mbox{ is finitely generated over } R_*\}.$$
\end{defn}

\begin{rem}
In other work by the first author this notion is called $c$-finite generation and the corresponding notion of regularity is called $c$-regularity (see Definition~\ref{defn:creg}). Due to the visual conflict with terminology we will introduce, namely the notion of finite generation with respect to a normalization, we expand the $c$ to coefficient throughout.
\end{rem}

 It is not clear that this class of objects  has good formal  properties unless the
coefficient ring $R_*$ is very nice. Nonetheless we will introduce a
better behaved  notion which appears to depend on additional data and
some of our main results will show that  in many cases that it agrees with the naive notion. 

The following concept is central to our analysis.

\begin{defn}\label{defn:g-normalization}
A {\em g-normalization} of $R\lra k$ is a map $q\colon S\lra R$ so that
$R$ and $k$ are small as $S$-modules, i.e.\ $S$ is g-regular and $q$ is relatively g-regular. 

Since there is no real possibility for confusion we will systematically omit the `g-' for brevity and refer to $q$ simply as a normalization.
\end{defn}

%We will henceforth omit the g-for brevity. 
This plays the role of Noether normalization in commutative algebra, and gives us a method for
defining an analogue of the bounded derived category. 

\begin{defn}\label{defn:q-fg}
Given a normalization $q$ as above, an $R$-module $M$ is said to be {\em $q$-finitely generated} if $q^* 
M$ is small over $S$. We define a corresponding thick subcategory
$$ \sfD^{q\mathrm{-b}}(R):=\DRfbbS :=\{ M\in \sfD(R) \st q^*M\finbuiltby S\}.$$
\end{defn}

If $R$ and $S$ are conventional Noetherian rings,  then an  $R$-module
is $q$-finitely generated if and only if its homology is finitely
generated in the conventional sense.  Accordingly the category $ \sfD^{q\mathrm{-b}}(R)$ is the analogue of the bounded derived category.

We will discuss the extent to which this depends on $q$ in Section~\ref{sec:dependence} and then again in Section~\ref{subsec:Gorfg}. 

\begin{rem}
There is an obvious small conflict in terminology between $g$-regular (and its relatives) and $q$-finitely generated. Our approach to this is to reserve the letter $g$ and only use it in the context of $g$-regularity and so on. 
\end{rem}

For now let us indicate what such normalizations look like in our examples.

\begin{ex}\label{ex:a2}   {\em (Algebra)}
Let $(R,\mathfrak{m},k)$ be a commutative Noetherian complete local $k$-algebra. By \cite{Cohen46}*{Theorem~16} we can find a subring $S$ of $R$, which is a power series ring, and over which $R$ is finite. This gives a normalization of $R$ and the above definition gives the usual bounded derived category of finitely generated modules.
\end{ex}

\begin{ex}   {\em (Rational homotopy theory)}
Returning to Example \ref{ex:dt}, with $R=C^*(X;\Q)$, 
 whenever $H^*(X)$ is finitely generated,  it is finite as a module over a polynomial
subring. We may then realize this polynomial subring by a map $X\lra
\prod_iK(\Q, 2n_i)$ which gives a normalization
\begin{displaymath}
q: S=C^*(\prod_iK(\Q, 2n_i))\lra R.
\end{displaymath}
We will see in Lemma
\ref{lem:creg} that this implies that an $R$-module $M$  is $q$-finitely generated if and only 
if it is coefficient-finitely generated (i.e.\  $H_*(M)$ is finitely generated over $H^*(X)$).
\end{ex}

\begin{ex}\label{ex:BGnormalization}   {\em (Representation theory)}
Returning to Example \ref{ex:g}, with $R=C^*(BG;\Fp)$, we may choose a
faithful representation $G\lra U(n)$. Then the map $S=C^*(BU(n))\lra
C^*(BG)=R$ is a normalization. Indeed,  $H^*(BU(n))$ is polynomial and 
by Venkov's theorem $H^*(BG)$ is finitely generated as a module
over it. Thus the cohomology of $BG$ has a finite projective
resolution over the cohomology of $U(n)$ and so by  Lemma \ref{lem:creg}
$C^*(BG)$ is finitely built from $C^*(BU(n))$, and  a $C^*(BG)$-module
$M$  is $q$-finitely generated if and only 
if it is coefficient-finitely generated (i.e.\  if and only if $H_*(M)$ is finitely generated over $H^*(BG)$).
\end{ex}

\subsection{Existence of normalizations}

We have just seen that in examples coming from rational homotopy
theory and representation theory that, not only do normalizations
exist, one can find normalizations of the same flavour, i.e.\ which
occur very naturally through some construction in that area. In this
subsection we work with classical associative DGAs, and  indicate how
one can construct normalizations in that context  rather generally. 
%Our result covers the aforementioned examples, but one does not produce a `natural' normalization.

We recall that a graded $k$-algebra $A$ is \emph{finitely presented} if it is a quotient of a finitely generated graded free algebra by a finitely generated homogeneous ideal, i.e.\ $A$ has a presentation of the form
\begin{displaymath}
A \simeq k\langle x_1,\ldots, x_n\rangle/I
\end{displaymath}
where $k\langle x_1,\ldots, x_n\rangle$ denotes the free algebra on the $n$ generators $x_1,\ldots, x_n$ and $I$ is a finitely generated ideal. 

\begin{thm}
Let $R\lra k$ be an augmented DG-algebra over $k$ such that $H^*(R)$ is a finitely presented $k$-algebra. Choose a presentation 
\begin{displaymath}
S= k\langle x_1,\ldots, x_n\rangle \stackrel{\pi}{\lra} H^*(R).
\end{displaymath}
Then $\pi$ can be lifted to a normalization $q\colon S \lra R$, where $S$ is viewed as a DG-algebra with trivial differential.
\end{thm}
\begin{proof}
Let $\pi\colon S\lra H^*(R)$ be as in the statement. We can lift $\pi$ to a map of DG-algebras $q\colon S\to R$ by choosing cocycles $z_i$ representing the generators $\pi(x_i)$ of $H^*(R)$. Indeed, the universal property of the free algebra gives a ring map $S\lra R$ sending $x_i$ to $z_i$, and since the $z_i$ are cocycles and $S$ has trivial differential this is a map of DG-algebras.

It remains to show that $q$ is a normalization, i.e.\ that $R$ and $k$ are small over $S$. By definition $H^*(R)$ is a finitely presented graded $S$-module and clearly so is $k$. It is thus enough to note that any DG-$S$-module with finitely presented cohomology is small; this follows from the fact that free algebras have global dimension $1$ (and in particular are coherent). One can prove this as in Proposition~\ref{prop:gcriterion}, which implies the statement for $k$ and whose proof generalizes to cover $R$. Alternatively, it is clear for finitely presented modules, since every ideal of a free algebra is a free module, and the usual argument shows that every DG-$S$-module is formal.
\end{proof}

%JG: This remark is correct but distracting. The point is that an
%associative algebra normalization does you no good if you are in a
%commutative algebra context. 
%\begin{rem}
%When $R$ is defined over a field the condition that $H^*(R)$ is finitely presented is certainly not necessary. %For instance any commutative regular local $k$-algebra is its own normalization, but these are typically not% finitely presented over $k$.
%\end{rem}

We are not aware of any restrictions imposed by the existence of a normalization in general, although presumably they are not for free. For instance, if $k$ is not finitely presented over $R_*$ then it seems very optimistic to expect a normalization to exist. At the very least this is an obstruction to the existence of normalizations with coherent homotopy (the relevant terminology is defined in Section~\ref{sec:dependence}).

\begin{prop}
Suppose that the augmentation $R\lra k$ is surjective on homotopy and $q\colon S\lra R$ is a normalization such that $S_*$ is a coherent ring. Then $k$ is finitely presented over $R_*$.
\end{prop}
\begin{proof}
The first part of the proof of Proposition~\ref{prop:Rcoherent} shows that, in this situation, $R_*$ is coherent, and the statement of said proposition shows that $k$ is cohomologically locally finitely presented over $R$. With this as input Lemma~\ref{lem_k_fp} tells us that $k$ is finitely presented over $R_*$.
\end{proof}

\section{Local finite presentation and dependence on normalization}
\label{sec:dependence}

We give a first discussion of  how the notions of finite generation and the bounded
derived category depend on the choice of normalization. We show that
in two situations they are independent of this choice. The first, Proposition~\ref{prop:cfgisf}, assumes the coefficient ring $S_*$ is regular, and the second, Corollary \ref{cor:sgfgislfp}, that it is coherent with a well
behaved derived category. The arguments proceed via the homological
algebra of cohomological functors.

These are useful criteria, but not sufficient to treat a general $g$-regular
ring. We will return to this question in Subsection
\ref{subsec:Gorfg}; we show that in our principal 
applications (where we have completeness and Gorenstein conditions)
 finite generation is independent of
the normalization. That argument is independent of those given here, so 
some readers may wish to skip this section.

\subsection{Modules over coefficient-regular rings}
\label{subsec:cregsmall}
We begin with a definition.

\begin{defn}\label{defn:creg}
We say that $S$ is \emph{coefficient-regular} if the coefficient
ring $S_*$ is a Noetherian regular ring. 
\end{defn}

This is a rather strong condition and implies $g$-regularity provided $k$ is finitely generated over $S_*$. 
In fact, if $S$ is coefficient-regular and $N$ is an $S$-module then
$N$ is small if and only if $N$ is coefficient-finitely generated, i.e.\ $N_*$ is finitely generated over
$S_*$ (see \cite{JohnNotes}*{Lemma~10.2} for a proof).

\subsection{Coefficient-regular normalizations}\label{ssec:crn}

If the normalization $S\stackrel{q}\lra R$ has the property that $S$ is
coefficient-regular then it is easy to understand when an $R$-module is
finitely generated. In this case we will call $q$ a \emph{coefficient-normalization}.

\begin{lem}\label{lem:creg}
If $S\stackrel{q}\lra R$ is a coefficient-normalization then an $R$-module $M$ is $q$-finitely
generated if and only if it is coefficient-finitely generated.
\end{lem}

\begin{proof}
By definition $M$ is $q$-finitely generated if and only if $q^*M$ is
small. Since $S$ is coefficient-regular, this happens if and only 
if $q^*M_*$ is finitely generated over $S_*$ (as noted above). Since  $R_*$ is finitely generated
as an $S_*$-module, $q^*M_*$ is a finitely generated $S_*$-module if
and only if $M_*$ is a finitely generated $R_*$-module. 
\end{proof}

\begin{prop}
\label{prop:cfgisf}
If $S\stackrel{q}\lra R$ is a coefficient-normalization then 
$$\DRfbbS=\sfD^f(R)=\{ M\st M_* \mbox{ is finitely generated over } R_*\}. $$
In particular, the left-hand side is independent of the chosen coefficient-normalization.\qed
\end{prop}

We will show in Subsection \ref{subsec:Gorfg} that the corresponding result
holds very generally for complete Gorenstein normalizations. 

\subsection{Locally finitely presented functors}

We next compare our definition to one coming from a more abstract notion of finiteness, namely that of being locally finitely presented. %, as is used for instance in \cite{Rdim}.

We fix a base commutative ring $A$ (for instance $\ZZ$). Let $\sfK$ be an $A$-linear triangulated category and let $F$ be an $A$-linear functor
\begin{displaymath}
F\colon \sfK^\op \to \AMod.
\end{displaymath}

\begin{defn}\label{def:lfp}
We say that $F$ is \emph{locally finitely generated} if for every $k\in \sfK$ there is an $l\in \sfK$ (allowed to depend upon $k$) and a natural transformation
\begin{displaymath}
\alpha \colon \sfK(-,l) \to F
\end{displaymath}
such that for all $i\in \ZZ$ the component
\begin{displaymath}
\alpha_{\Sigma^i k}\colon \sfK(\Sigma^i k,l) \to F(\Sigma^i k)
\end{displaymath}
is surjective.

We say $F$ is \emph{locally finitely presented} if it is locally finitely generated and for any natural transformation $\sfK(-,l)\to F$ the kernel, taken in the functor category, is again locally finitely generated.
\end{defn}

Following Rouquier \cite{Rdim} it is convenient to formulate being locally finitely presented in the following slightly more tractable fashion. Given a functor $F$ and an object $k\in \sfK$ we can consider the conditions:
\begin{itemize}
\item[(a)] there is an $l\in \sfK$ and an $\alpha\colon \sfK(-,l)\to F$ such that $\alpha_{\Sigma^i k}$ is surjective for all $i\in \ZZ$;
\item[(b)] for every $\beta\colon \sfK(-,m) \to F$ there is an $f\colon n\to m$ such that \mbox{$\beta\circ \sfK(-,f) = 0$} and
\begin{displaymath}
\xymatrix{
\sfK(\Sigma^i k, n) \ar[rr]^-{\sfK(\Sigma^i k, f)} && \sfK(\Sigma^i k, m) \ar[r]^{\beta_{\Sigma^i k}} & F(\Sigma^i k)
}
\end{displaymath}
is exact for each $i\in \ZZ$.
\end{itemize}
It is straightforward to check that $F$ is locally finitely presented if and only if it satisfies conditions (a) and (b) for every object of $\sfK$. 

Now let us fix a triangulated category $\sfT$ with small coproducts and a generating set of small objects (i.e.\ $\sfT$ is compactly generated) and let $\sfT^\mathrm{c}$ denote the thick subcategory of small objects. Our main interest in Definition~\ref{def:lfp} is that it provides a very natural class of objects in $\sfT$ which is intrinsically defined (via the compact objects).

\begin{defn}\label{defn:clfp}
We say an object $X$ of $\sfT$ is \emph{cohomologically locally finitely generated} (respectively \emph{presented}) if the functor it represents when restricted to $\sfT^\mathrm{c}$ is locally finitely generated (respectively presented), i.e.\ $\sfT(-,X)\vert_{\sfT^\mathrm{c}}$ is locally finitely generated (presented).
\end{defn}

We denote by $\sfT^{\mathrm{lfp}}$ the full subcategory of cohomologically locally finitely presented objects and recall from \cite{Rdim}*{Proposition~4.28} that it is a thick subcategory of $\sfT$. Setting $\sfT = \sfD(R)$, this gives another candidate for the bounded derived category of a ring spectrum (which has the benefit of making sense in more abstract contexts). 

\subsection{Coherent classical generators}

In this section we again fix a triangulated category $\sfK$, over some base ring $A$, which we assume for simplicity is idempotent complete. We will assume $\sfK$ has a classical generator $g$, i.e. there is an equality
\begin{displaymath}
\sfK = \thick(g).
\end{displaymath}
Put yet another way we have $g \finbuilds k$ for every $k\in \sfK$. We can make the generation process a bit more explicit as follows. We define $\langle g \rangle_1$ to be the closure of $\{\Sigma^ig\; \vert \; i\in \ZZ\}$ under finite direct sums and summands. We then inductively define $\langle g \rangle_{i+1}$ to be the full subcategory of $\sfK$ consisting of those objects $k$ for which there is a $k'$ and a triangle
\begin{displaymath}
l \to k\oplus k' \to m \to \Sigma l
\end{displaymath}
with $l\in \langle g \rangle_i$ and $m\in \langle g \rangle_1$. Thus
$\langle g \rangle_{i+1}$ consists of those objects which $g$ builds
by taking at most $i$ cones. The above makes sense for any object of
$\sfK$ and the statement that $\sfK = \thick(g)$ just says the union
of the $\langle g \rangle_i$ is $\sfK$. 

Given objects $k$ and $k'$ in $\sfK$ we set
\begin{displaymath}
\sfK^*(k,k') = \bigoplus_{i\in \ZZ} \sfK(k, \Sigma^ik').
\end{displaymath}

Recall that an additive functor $F\colon \sfK^\op \to \AMod$ is \emph{cohomological} if it sends triangles to long exact sequences. 

\begin{lem}\label{lem_thick_fp}
Let $k$ be an object of $\sfK$ and suppose that $\sfK^*(k,k)$ is a coherent graded ring. If $l\in \thick(k)$ then $\sfK^*(k,l)$ is a finitely presented $\sfK^*(k,k)$-module.
\end{lem}
\begin{proof}
We proceed by induction on the number of cones required to build $l$ from $k$. If $l\in \langle k \rangle_1$ then the statement is clear. Suppose then that the statement holds for objects of $\langle k \rangle_{i-1}$ and let $l\in \langle k \rangle_{i}$. By definition there is a triangle
\begin{displaymath}
m \to l' \to n \to \Sigma m
\end{displaymath}
with $m\in \langle k \rangle_{i-1}$, $n\in \langle k \rangle_1$ and $l$ a summand of $l'$. This triangle gives rise to an exact sequence of graded modules
\begin{displaymath}
\sfK^*(k,\Sigma^{-1}n) \to \sfK^*(k,m) \to \sfK^*(k,l') \to \sfK^*(k,n) \to \sfK^*(k,\Sigma m).
\end{displaymath}
By the induction hypothesis all but the middle term are finitely presented and it follows, from coherence of $\sfK^*(k,k)$, that $\sfK^*(k,l')$ is also finitely presented. It is then clear that $\sfK^*(k, l)$ is also finitely presented as required.
\end{proof}

\begin{prop}\label{prop_clfp}
Suppose that $\sfK = \thick(g)$ as above and that, in addition, $\sfK^*(g,g)$ is a coherent graded ring. Then a cohomological functor $F$ on $\sfK$ is locally finitely presented if and only if $\bigoplus_{i\in \ZZ}F(\Sigma^i g)$ is finitely presented over $\sfK^*(g,g)$.
\end{prop}
\begin{proof}
Suppose first that $F$ is locally finitely presented. Then by conditions (a) and (b) at $g$ there are natural transformations
\begin{displaymath}
\sfK(-,m) \to \sfK(-,l) \to F
\end{displaymath}
such that the sequence of $\sfK^*(g,g)$-modules
\begin{displaymath}
\sfK^*(g,m) \to \sfK^*(g,l) \to \bigoplus_{i\in \ZZ}F(\Sigma^i g) \to 0
\end{displaymath}
is exact. By the previous lemma, using that $g$ classically generates, the first two terms of this sequence are finitely presented and thus so is the cokernel.

On the other hand, let us suppose that $\bigoplus_{i\in \ZZ}F(\Sigma^i g)$ is a finitely presented $\sfK^*(g,g)$-module. By \cite{Rdim}*{Lemma~4.6} it is enough to check conditions (a) and (b) at the object $g$. Condition (a) is clear as we can just pick a finitely generated graded free module mapping onto $\bigoplus_{i\in \ZZ}F(\Sigma^i g)$ and Yoneda gives us the desired natural transformation.

Suppose we are given, with a view to verifying (b), a natural transformation
\begin{displaymath}
f\colon \sfK(-,m) \to F.
\end{displaymath}
Then, since $\sfK^*(g,g)$ is coherent, the module $(\ker f)\vert_{\{\Sigma^i g \; \vert \; i\in \ZZ\}}$ is finitely presented by virtue of being the kernel of a map between finitely presented modules. Thus using (a) for the kernel we can produce the sequence required in (b). 
\end{proof}

\subsection{A criterion for g-regularity}

Now let us again return to our standard setting of a fixed ring
spectrum $R$ with an augmentation to a field $k$. In this section,
which is somewhat of an aside, we give a criterion for $R$ to be
g-regular in terms of strong generation of the full subcategory of
small $R$-modules. Recall that $\sfD^\mathrm{c}(R)=\DRfbbR$. In this context cohomologically locally finitely presented will always mean with respect to the small modules. We denote by $\sfD^\mathrm{lfp}(R)$ the thick subcategory of cohomologically locally finitely presented modules.

Throughout this section we will assume the augmentation $R\to k$ is surjective on homotopy, i.e.\ $R_* \to k$ is a surjection. We will denote by $I$ the ``augmentation ideal'' which is defined by the triangle
\begin{displaymath}
I \to R\to k \to \Sigma I
\end{displaymath}
and has homotopy the usual graded augmentation ideal $I_* = \ker(R_* \to k)$.

\begin{lem}\label{lem_k_fp}
Suppose that $R_*$ is coherent. Then the following are equivalent:
\begin{itemize}
\item[$(1)$] $k$ is cohomologically locally finitely presented in $\sfD(R)$;
\item[$(2)$] $k$ is a finitely presented $R_*$-module;
\item[$(3)$] $I_*$ is finitely generated as a $R_*$-module.
\end{itemize}
\end{lem}
\begin{proof}
Since $R_*$ is coherent and $R$ classically generates $\sfD^\mathrm{c}(R)$ the statement that (1) holds if and only if (2) holds is just Proposition~\ref{prop_clfp}. That (2) and (3) are equivalent is just the definition of finite presentation. 
\end{proof}

We recall that a triangulated category $\sfK$ is called {\em strongly 
  generated} if there is an object $g$ and an $n$ for which 
$\sfK=\langle g\rangle_n$. This is a somewhat restrictive condition: the finite stable homotopy category is not strongly generated, and if the category of perfect complexes over a finitely generated $k$-algebra $A$ is strongly generated then every finitely generated $A$-module has finite projective dimension \cite{OS12}. On the other hand, the bounded derived category of a noetherian $k$-algebra is known to be strongly generated in many examples \cite{elagin2018smoothness, neeman2017strong} and the category of perfect complexes over a homologically smooth DG-algebra is strongly generated \cite[Lemmas 3.5, 3.6]{Lunts}.

\begin{prop}\label{prop:gcriterion}
Suppose that $R_*$ is coherent and $I_*$ is a finitely generated
$R_*$-module. If $\sfD^\mathrm{c}(R)$ is strongly generated, then the ring spectrum $R$ is g-regular.
\end{prop}
\begin{proof}
By the lemma $k$ is a cohomologically locally finitely presented object of $\sfD(R)$. As $\sfD^\mathrm{c}(R)$ is strongly generated the representability theorem \cite{Rdim}*{Theorem~4.16} applies and tells us that in fact $k\in \sfD^\mathrm{c}(R)$, i.e.\ we have $R\finbuilds k$; this is nothing other than the definition of g-regularity of $R$.
\end{proof}

\subsection{Smooth coherent normalizations}

We now compare the definition we have given of the bounded derived
category, relative to a normalization, in Section~\ref{sec:Db} to
the category of cohomologically locally finitely presented objects. 
 We prove the following theorem.

\begin{thm}
\label{thm:blfp}
Let $q\colon S\to R$ be a normalization of $R\to k$. If $S_*$ is coherent then there is a containment
\begin{displaymath}
\sfD^{q\mathrm{-b}}(R) \subseteq \sfD^\mathrm{lfp}(R).
\end{displaymath}
Moreover, if $\sfD^\mathrm{c}(S)$ is strongly generated this containment is an equality.
\end{thm}

As a consequence we obtain, at least under mild assumptions, another invariance result for our definition of the bounded derived category.

\begin{cor}
\label{cor:sgfgislfp}
Suppose $S\stackrel{q}{\to}R$ and $S'\stackrel{q'}{\to}R$ are normalizations of $R$ with $S_*$ and $S'_*$ coherent and both $\sfD^\mathrm{c}(S)$ and $\sfD^\mathrm{c}(S')$ strongly generated. Then
\begin{displaymath}
\sfD^{q\mathrm{-b}}(R) = \sfD^{q'\mathrm{-b}}(R)
\end{displaymath}
as thick subcategories of $\sfD(R)$.
\end{cor}

We begin by proving the containment that always holds.

\begin{prop}\label{prop:Rcoherent}
Let $q\colon S\to R$ be a normalization of $R$ and assume that $S_*$ is coherent. If $X$ in $\sfD(R)$ is $q$-finitely generated, i.e.\ $q^*X$ is small over $S$, then $X$ is cohomologically locally finitely presented over $R$, i.e.\
\begin{displaymath}
\sfD^{q\mathrm{-b}}(R) \subseteq \sfD^\mathrm{lfp}(R).
\end{displaymath}
\end{prop}
\begin{proof}
Suppose that $q^*X$ lies in $\thick(S)$. Then $(q^*X)_* \cong X_*$ is a finitely presented $S_*$-module by Lemma~\ref{lem_thick_fp}. In particular, since $S\to R$ is a normalization, we can take $X=R$ to see that the $S_*$-module $R_*$ is finitely presented. In particular, the ring $R_*$ is also coherent.

%As $R_*$ is finitely presented over $S_*$ we deduce, from finite presentation of $X_*$ over $S_*$, that $X_*$ is a finitely presented $R_*$-module. 
We now prove that $X_*$ is a finitely presented $R_*$-module. Since $X_*$ is finitely presented over $S_*$ it follows that $X_*\otimes_{S_*} R_*$ is finitely presented over $R_*$. This latter module has $X_*$ as a quotient and so certainly $X_*$ is finitely generated. We conclude by choosing a surjection $R_*^{\oplus n} \to X_*$ and noting that, since $S_*$ is coherent, the kernel of this surjection is finitely presented over $S_*$ and hence finitely generated by the argument we have just given.

Using finite presentation of $X_*$ and coherence of $R_*$ we can apply Proposition~\ref{prop_clfp} which tells us that $X$ is cohomologically locally finitely presented in $\sfD(R)$.
\end{proof}

We now prove the reverse containment under the strong generation hypothesis.

\begin{prop}
Let $q\colon S\to R$ be a normalization of $R$ such that $S_*$ is coherent and $\sfD^\mathrm{c}(S)$ is strongly generated. Then if $X\in \sfD(R)$ is cohomologically locally finitely presented the module $q^*X$ is small over $S$, i.e.\
\begin{displaymath}
\sfD^{q\mathrm{-b}}(R) \supseteq \sfD^\mathrm{lfp}(R).
\end{displaymath}
\end{prop}
\begin{proof}
Suppose $X\in \sfD^\mathrm{lfp}(R)$ as in the statement. As in the proof of the previous proposition we can use Lemma~\ref{lem_thick_fp} to see that $R_*$ is finitely presented over $S_*$ and so coherence of $S_*$ implies coherence of $R_*$. Thus we can apply Proposition~\ref{prop_clfp} to see that $X_*$ is finitely presented over $R_*$. 

Using again that $R_*$ is finitely presented over $S_*$ this tells us that $X_*$ is a finitely presented $S_*$-module. Given that we have assumed $S_*$ coherent we may then apply Proposition~\ref{prop_clfp} to deduce that $q^*X$ is cohomologically locally finitely presented in $\sfD(S)$. The assumption that $\sfD^\mathrm{c}(S)$ is strongly generated then implies, by virtue of \cite{Rdim}*{Theorem~4.16}, that $q^*X$ is actually small.
\end{proof}

\section{The Six Ring Context}\label{sec:standardcontext}
The starting point of our analysis is a chosen normalization of a
`local ring' $R\lra k$. We show here that this gives rise to two Koszul dual
cofibre sequences of rings, which will provide the framework for our further 
results.

\subsection{The set-up}\label{sec:set-up}
We suppose we are given maps $S \stackrel{q}\lra R \lra k$ of ring
spectra with $k$ a field. We write $Q=R\tensor_Sk$ for the cofibre of $q$. We will assume from here on that $R$ and $k$ are small as $S$-modules. Thus $S$ is g-regular, and $q$ is a normalization of $R$. 

\begin{lem}\label{lem:proxy}
Under the above assumptions, $R$ is proxy-regular, i.e.\ $k$ is proxy-small over $R$, and $Q$ can be taken as a proxy for $k$.
\end{lem}

\begin{proof}
Since $Q$ is a ring and $k$ is a module over it we have $Q \builds k$. For the other two conditions, we use the fact that both $k$ and
$R$ are small over $S$:
$$\left( S\finbuilds R \right) \Rightarrow \left( k\simeq S\tensor_S k
\finbuilds R\tensor_S k = Q\right)$$
and 
$$\left( S\finbuilds k \right) \Rightarrow \left( R\simeq R\tensor_S S
\finbuilds R\tensor_S k = Q\right).$$
\end{proof}

\subsection{The Koszul dual cofibre sequence}
The Koszul duals of the rings
$$S\stackrel{q}\lra R \stackrel{p}\lra Q$$
are the rings
$$\cF \stackrel{j}\lla \cE\stackrel{i}\lla \mcD$$
where 
$$\cF=\Hom_S(k,k),\; \cE=\Hom_R(k,k), \mbox{ and } \mcD=\Hom_Q(k,k). $$

\begin{lem}
\label{lem:MoritaRsmall}
The sequence
$$\cF \lla \cE\lla \mcD$$
is also  a cofibre sequence. Moreover, $\cF$ has finite dimensional homotopy over $k$ and $\cF$ is small
over $\cE$.
\end{lem}
\begin{rem}
Topologists may think of the example arising from a fibration
$$Y\lla X\lla F$$
with $S=C^*(Y)$, $R=C^*(X)$, so that provided the Eilenberg-Moore
spectral sequence converges (e.g. if  $Y$ is 1-connected \cite[Theorem
7.1]{McCleary}), $Q=R\tensor_S k\simeq C^*(F)$. We see that the condition that $R$
is small over $S$ is equivalent to the condition that $H^*(F)$ is finite dimensional.

Continuing the fibre sequence we obtain
$$Y\lla X\lla F\lla \Omega Y\lla \Omega X \lla \Omega F. $$
Provided the Eilenberg-Moore spectral sequences converge, we find $\cF\simeq C_*(\Omega Y)$, $\cE \simeq C_*(\Omega X)$ and $\mcD \simeq C_*(\Omega
F)$.  Again, the condition that $S$ is g-regular is the condition that $H_*(\Omega
Y)$ is finite dimensional, and the condition that $\cF$ is small over $\cE$ is that
$H_*(F)$ is finite dimensional. 
\end{rem}

\begin{proof}[Proof of Lemma~\ref{lem:MoritaRsmall}.]
First we show that $\mcD \lra \mcE \lra \mcF$ is a cofibre sequence,
which is to say that 
$$\mcF \simeq \mcE\tensor_{\mcD} k. $$
Expanding the definition of the right hand side 
\begin{align*}
\mcE\tensor_{\mcD} k &=\HomR (k,k)\tensor_{\HomQ(k,k)} k \\
&\simeq \HomQ (k,\HomR(Q,k))\tensor_{\HomQ(k,k)} \HomQ(Q,k).
\end{align*}
In general, for a $Q$-module $L$, 
 composition gives a map 
$$\HomQ (k, \HomR (Q,k))\tensor_{\HomQ(k,k)} \HomQ(L,k)\lra
\HomQ(L,\HomR(Q,k)),$$
where the target can be identified with $\HomR(L,k)$ by adjunction. 
This map is obviously an equivalence when $L=k$, and hence for any
$Q$-module $L$ (such as $Q$) finitely built from $k$. Taking $L=Q$,
 we have
\begin{multline*}
\HomQ (k,\HomR(Q,k))\tensor_{\HomQ(k,k)} \HomQ(Q,k)\simeq \HomR (Q,
k)\\
 =\HomR(R\tensor_S k,k)\simeq \HomS(k,k)=\cF
\end{multline*}
as required.

By definition, g-regularity of $S$ means that $\cF = \Hom_S(k,k)$ is finite
dimensional. Finally, we  show that $\mcF $ is small over $\mcE$.
 Clearly
$$\left( S\finbuilds R\right) \Rightarrow \left( k\simeq S\tensor_S k\finbuilds
R\tensor_Sk\right). $$
Applying $\Hom_R(\cdot , k)$ we find
$$\cE =\Hom_R(k, k )\finbuilds \Hom_R(R\tensor_Sk , k)\simeq \Hom_S(k,k)=\cF.$$
\end{proof}

\subsection{Another criterion for g-regularity}
The following application of Thomason's Localization Theorem is
straightforward but amusing. The version of the Localization Theorem
to which we appeal is due to Neeman \cite{NeemanLoc}, but the most
convenient version for our purposes is  \cite[Theorem 2.1]{NeeGrot}.

\begin{lem}
\label{lem:Thomason}
If $\mcF \vdash \mcE$ then the completion of $R$, $\Hom_{\cE}(k,k)$, is g-regular. 
\end{lem}
\begin{proof}
We suppose that $\mcF \vdash \mcE$. By Lemma \ref{lem:MoritaRsmall} $\mcF$ is small over $\mcE$ so we deduce, via Thomason's Localization Theorem, that in fact $\mcF \vDash \mcE$. Since $S$ is g-regular we know $\mcF$ is finite dimensional from which we conclude
\begin{displaymath}
k \vDash \mcF \vDash \mcE.
\end{displaymath}
Hence $\mcE$ is also finite dimensional, and applying
$\Hom_{\cE}(\cdot, k)$ to $k\finbuilds \cE$ we see the completion of
$R$  is g-regular.
\end{proof}

This style of argument will appear again in Proposition~\ref{prop:invariance} and the results following it where we deduce another general invariance statement for our notion of the bounded derived category.

\section{The Symmetric Gorenstein Context}\label{sec:SGC}

We continue with the notation and hypotheses of Section~\ref{sec:set-up}. From the normalization $S\lra R $ we have produced cofibre sequences
$$S\stackrel{q}\lra R\stackrel{p} \lra Q \quad \text{and} \quad \mcF \stackrel{j}\lla \mcE \stackrel{i}\lla \mcD,$$
the latter being the Koszul dual of the former.

Concentrating on $R$, there is a functor $E$ from right $R$-modules to right
$\cE$-modules given by
$$EM=\Hom_R(k, M).$$ There are similar comparison functors relating modules over $S$ and $Q$ to $\mcF$ and $\mcD$ respectively.
To complete our comparison, we
need to be able to
return  from the second cofibre sequence to the first. Accordingly, we need a
suitable {\em right} $\cE$-module structure on $k$, and we will
therefore assume the Gorenstein condition at various points. We show that
this rather elaborate structure occurs remarkably often and leads to a
rich network of related functors.

\subsection{Gorenstein}\label{ssec:Gorenstein}
The usual definition of a commutative Gorenstein local ring $(R, \fm , k)$ is that
$R$ is of finite injective dimension as  a module over itself, but one
then proves that this is equivalent to saying $\Ext^*_R(k,R)$ is one dimensional over
$k$. It is the latter condition that we use to extend the definition
to our context \cite{DGI2}.

\begin{defn}\label{defn:gor}
A map $R\lra k$ is said to be {\em Gorenstein of shift $a_R$} if there is some weak equivalence 
$\HomR(k,R)\simeq \Sigma^{a_R}k$ of $R$-$R$-bimodules. 

More generally, a map $q\colon S\lra R$ is said to be
{\em relatively Gorenstein of shift $a_q $} if $\HomS(R,S)\simeq
\Sigma^{a_q}R$ as $S$-$S$-bimodules.
\end{defn}

\subsection{The condition}
The basic structure behind our results may be summarized as follows.

\begin{defn}\label{defn:sgc}
We say that a cofibre sequence $S\stackrel{q}\lra R\stackrel{p}\lra Q$ and its Koszul dual
$\cF \stackrel{j}\lla \cE\stackrel{i}\lla \mcD$ form a {\em Symmetric Gorenstein Context} if 
\begin{itemize}
\item all six ring spectra  are Gorenstein;
\item all four maps $p, q, i$, and $j$ are relatively Gorenstein;
\item the two  rings $S$ and $\mcD$ are g-regular;
\item all four maps $p, q, i$, and $j$ are relatively g-regular (see Section~\ref{sec:regularity}).
\end{itemize}
Informally, we may say it is $6+4$ Gorenstein and $2+4$ g-regular. 
\end{defn}

\subsection{From normalization to the Symmetric Gorenstein Context}
The number of conditions in the definition of a Symmetric Gorenstein
Context looks daunting. However, we show that the whole structure
can be deduced from appropriate conditions on the original normalization $q\colon S
\lra R$. 

\begin{defn}\label{defn:sgn}
A map of augmented ring spectra $q\colon S \lra R$ is a \emph{strongly Gorenstein normalization} if
\begin{itemize}
\item $S$ is Gorenstein and  $S\lra R$ is relatively Gorenstein and
\item $S$ is g-regular and $S\lra R$ is relatively g-regular 
\end{itemize}
\end{defn}

\begin{prop}
\label{prop:halfsymmetric}
Suppose that $q\colon S\lra R$ is a strongly Gorenstein normalization. Then $S\stackrel{q}{\lra} R\stackrel{p}{\lra} Q$ has all the properties required of it in a
Symmetric Gorenstein Context.
\end{prop}

\begin{rem}
Informally $1+1$ Gorenstein and $1+1$ g-regular implies $3+2$ Gorenstein
and $1+2 $ g-regular. 
\end{rem}

We will repeatedly use the observation that one has Gorenstein ascent and descent along
relatively Gorenstein maps. 

\begin{lem}
\label{lem:relGorascdesc}
If  $f\colon B\lra A$ is relatively Gorenstein
then $A$ is Gorenstein if and only if $B$ is Gorenstein, and if these
hold then $a_A+a_f=a_B$.
\end{lem}
\begin{proof}
We have the equivalences 
$$\Hom_A(k,A)\simeq \Hom_A(k, \Sigma^{-a_f}\Hom_B(A,B))\simeq
\Sigma^{-a_f}\Hom_B(k,B). $$
\end{proof}

\begin{proof}[Proof of Proposition \ref{prop:halfsymmetric}.]
The required regularity statements are that $S$ is g-regular and the maps $q$
and $p$ are relatively g-regular. The first two are hypotheses. For the
third,  since  $k$ is $S$-small  $Q=R\otimes_S k$ is $R$-small.

The required Gorenstein statements are  that $S, R$ and $Q$ are
Gorenstein, and that  $q$ and $p$ are relatively Gorenstein. 
Since $q$ is relatively Gorenstein, the fact that $R$ is Gorenstein follows by ascent from
the fact  $S$ is Gorenstein. 

For $p$ we make the computation
\begin{align*}
\Hom_R(Q,R) &\simeq \Hom_R(R\otimes_S k, R) \\
&\simeq \Hom_S(k, R) \\
&\simeq R \otimes_S \Hom_S(k,S)\\
&\simeq R\otimes_S \S^{a_S}k\\
&\simeq \S^{a_S}Q
\end{align*}
where the third isomorphism uses that $k$ is small over $S$ and the fourth that $S$ is Gorenstein (of shift $a_S$). That $Q$ is Gorenstein then follows by ascent from the fact that $R$ is Gorenstein. 
\end{proof}

Let us now consider the corresponding conditions on $\mcD, \mcE$ and
$\mcF$. We make the additional assumption that at least one of $\mcF, \mcE$ or $\mcD$ 
is Gorenstein.

\begin{prop}
\label{prop:fullsymmetric}
Suppose $S\lra R$ is a strongly Gorenstein normalization and that in 
addition at least one of $\mcD, \cE$ or $\mcF$ is Gorenstein, then we have a Symmetric Gorenstein Context. 
\end{prop}

\begin{rem}
Informally $1+1+1'$ Gorenstein and $1+1$ g-regular implies $6+4$ Gorenstein 
and $2+4 $ g-regular. 
\end{rem}

The additional assumption is often automatic: if $R$ is a $k$-algebra, proxy-regular and complete (see Section~\ref{sec:completion}) then $\cE$ is Gorenstein 
\cite[8.5]{DGI2}. 

\begin{cor} 
\label{rem:MoritaGorenstein}
Suppose $S\lra R$ is a strongly Gorenstein normalization and that $R$
is a $k$-algebra, proxy-regular and complete. Then we have a Symmetric Gorenstein Context. 
\end{cor}

\begin{rem}
Informally $1+1$ Gorenstein and $1+1$ g-regular implies $6+4$ Gorenstein 
and $2+4 $ g-regular. 
\end{rem}

\begin{proof}[Proof of Proposition \ref{prop:fullsymmetric}.]
We saw in Proposition \ref{prop:halfsymmetric} that $S\stackrel{q}{\lra} R\stackrel{p}{\lra} Q$ has all the properties
required, so we consider the properties of $\mcD\stackrel{i}{\lra} \mcE\stackrel{j}{\lra} \mcF$. 

We begin with the regularity properties. We showed in Lemma
\ref{lem:MoritaRsmall} that  $\mcF$ is small over $\mcE$. 
It is easy to see that as  $\mcD$-modules $\mcE$  and $k$ are small: for $k$ we note that $k\finbuilds Q$ and apply $\HomQ(\cdot , k)$. For $\mcE$, we note that  Proposition \ref{prop:halfsymmetric} proves
$R\finbuilds Q$. Applying $-\otimes_R k$ we see
\begin{displaymath}
k \finbuilds Q\otimes_R k.
\end{displaymath}
An application of $\Hom_Q(-,k)$ then yields
\begin{displaymath}
\mcD = \Hom_Q(k,k) \vDash \Hom_Q(Q\otimes_R k, k) \simeq \mcE.
\end{displaymath}

Finally,  we turn to the Gorenstein properties. 
Since we are assuming that at least one of $\mcD, \mcE$  or $\mcF$ is
Gorenstein, in view of Lemma \ref{lem:relGorascdesc} 
it suffices to show that $i$ and $j$ are relatively
Gorenstein. This is the content of Lemmas~\ref{lem:jrelGor} and \ref{lem:irelGor}.

\begin{lem}
\label{lem:jrelGor}
Suppose that $q\colon S\lra R$ is a strongly Gorenstein normalization. Then the dual map $\cE\lra \cF$ is relatively Gorenstein of shift $-a_q$ i.e.\
$$\Sigma^{a_q}\Hom_{\cE}(\cF, \cE) \simeq \cF.$$
\end{lem}

\begin{proof} 

Let us write $E$ for the functor to $\mcE$-modules defined by
$$EM=\Hom_R(k, M).$$
We first observe that $\cF\simeq \Sigma^{a_q}E(R\tensor_S k)$.  
Indeed, 
$$\cF =\Hom_S(k, k )\simeq \Hom_R(k, \Hom_S(R,k))=E(\Hom_S(R,k)). $$ 
Since $R$ is small over $S$, 
$$\Hom_S(R,k)\simeq \Hom_S(R, S)\tensor_S k\simeq
\Sigma^{a_q}R\tensor_Sk. $$
 
Now note that the map 
$$E\colon \Hom_R(T, k) \lra \Hom_{\cE}(ET, Ek)$$ 
is an equivalence for $T=k$ and hence if $T$ is finitely built from
$k$. In particular,  since $R$ is small over $S$, it applies to
$T=R\tensor_Sk$ to give
$$\Hom_S(k,k)\simeq
\Hom_R(R\tensor_Sk, k)\simeq \Hom_{\cE}(E(R\tensor_Sk), Ek)=\Sigma^{a_q}\Hom_{\cE}(\cF,\cE),$$
i.e.\ we have demonstrated the Gorenstein condition $\Sigma^{-a_q}\cF= \Hom_{\cE}(\cF,\cE)$. 
\end{proof}

The proof for $i$ is rather similar. 
%{\bf [[Perhaps it can be cut down or omitted]]}
\begin{lem}
\label{lem:irelGor}
Suppose that $q\colon S\lra R$ is a strongly Gorenstein normalization, so in particular the map $p\colon R\lra Q$ is also relatively Gorenstein, of shift $a_p$. The map $i\colon \mcD \to \mcE$, which is dual to $p$, is relatively Gorenstein of shift $-a_p$.
\end{lem}

\begin{proof}
First observe that since $Q$ is $R$-small we have
\begin{displaymath}
\Hom_R(Q,k) \simeq \Hom_R(Q,R)\otimes_R k \simeq \S^{a_p}Q \otimes_R k.
\end{displaymath}
Thus, writing $D$ for the functor to $\mcD$-modules defined by $DL=\Hom_Q(k,L)$, we can find
\begin{displaymath}
\mcE = \Hom_R(k,k) \simeq \Hom_Q(k, \Hom_R(Q,k)) = D(\Hom_R(Q,k)) \simeq D(\S^{a_p}Q\otimes_R k).
\end{displaymath}
Next we observe that, for $T\in \Modu Q$, the map
\begin{displaymath}
D\colon \Hom_Q(T,k) \to \Hom_\mcD(DT, Dk)
\end{displaymath}
is an equivalence for $T=k$ and hence is an equivalence for any $T$
finitely built from $k$. Since $R\vDash Q$ we see
\begin{displaymath}
k \simeq  R\otimes_R k \vDash Q\otimes_R k
\end{displaymath}
so this includes $T=Q\tensor_Rk$. We may therefore calculate
\begin{align*}
\mcE &= \Hom_R(k,k) \\
&\simeq \Hom_Q(Q\otimes_R k, k) \\
&\simeq  \Hom_\mcD(D(Q\otimes_R k), Dk) \\
&\simeq  \Hom_\mcD(\Sigma^{-a_p}\mcE, \mcD) \\
&\simeq \S^{a_p}\Hom_\mcD(\mcE, \mcD).
\end{align*}
\end{proof}
This completes the proof that we have a Symmetric Gorenstein Context. 
\end{proof}

\subsection{Examples from commutative algebra}
We could take $R$ to be a commutative Noetherian complete local $k$-algebra with residue field $k$, cf.\ Example~\ref{ex:a2}. Inside of $R$ we can find a power series ring $S$, with $R$ a finitely generated module over $S$. The ring $S$ is regular and so $R$ is small over $S$. Accordingly we have $1+1$ g-regularity, and $S$ is Gorenstein since it is an honest commutative regular ring. Finally, we must assume in addition that $S\lra R$
is relatively Gorenstein.  

In fact, it is enough to assume the cofibre $Q$ is Gorenstein. Indeed, if this is the case then $R$ is
  Gorenstein by Gorenstein ascent (as in \cite{AFHdescent}*{Theorem~4.3.2}, see also \cite{DGI2}*{Proposition~8.6}). It then follows
from the Auslander-Buchsbaum formula, together with the fact that $R$ is Cohen-Macaulay, that $R$ is free as an $S$-module. It is an immediate consequence that $S\lra R$ is relatively Gorenstein. This shows that $S\lra R$ being relatively Gorenstein is equivalent to $Q$
being Gorenstein as claimed. Since $R$ and $S$ are complete they are also complete in the sense defined in \ref{ssec:completion} by
\cite[4.20]{DGI2} and hence Corollary~\ref{rem:MoritaGorenstein} applies to show we have Symmetric Gorenstein Context.

\subsection{Examples from Koszul duality}\label{ex:Koszul}
We could take for $R = \Lambda$ a Gorenstein Koszul algebra of finite global dimension viewed as a formal DGA. Since $\Lambda$ is already regular we can also take $S = \Lambda$ and then the cofibre $Q$ is simply $k$. Clearly the identity map is relatively Gorenstein and so either by Proposition~\ref{prop:fullsymmetric} or inspection we get a Symmetric Gorenstein Context consisting of cofibre sequences
\begin{displaymath}
\Lambda \stackrel{1}{\lra} \Lambda \lra k
\end{displaymath}
and
\begin{displaymath}
\Lambda^! \stackrel{1}{\lla} \Lambda^! \lla k
\end{displaymath}
where $\Lambda^!$ is the Koszul dual viewed as a formal DGA, and $k$ is a normalization of $\Lambda^!$ by virtue of the latter being finite dimensional.

\begin{comment}
\subsection{Examples from Ginzburg dg-algebras}
Returning to Example~\ref{ex:Ginzburg} we are in a similar
situation. Given a quiver $Q$ with potential $w$ (as in
\ref{ex:Ginzburg}) and letting $l = kQ/\rad(kQ)$ we have a 
 Symmetric Gorenstein Context (permitting the slight  generalisation
 that $l$ is not a field)
\begin{displaymath}
\Gamma(Q,w) \stackrel{1}{\lra} \Gamma(Q,w) \lra l
\end{displaymath}
and
\begin{displaymath}
\cE \stackrel{1}{\lla} \cE \lla l
\end{displaymath}
where, as in the Koszul duality example, $\Gamma(Q,w)$ is regular and so its Morita partner $\cE$ has finite dimensional homotopy.
\end{comment}

\subsection{Examples from rational homotopy theory}
As in Example \ref{ex:dt} we take $R=C^*(X;\Q)$ and $k=\Q$. If we
suppose that $H^*(X; \Q)$ is Noetherian we may choose a polynomial
subring on even generators over which it is a finitely generated module. Take 
$B$ to be the corresponding product of even Eilenberg-MacLane spaces 
and $X\lra B$ realizing the inclusion of this polynomial subring, with
fibre $F$. We then set $S=C^*(B)$ and can identify the cofibre $Q$ with $C^*(F)$, which has finite homology.
This gives $1+1$ g-regular, and that $S$ is Gorenstein. We also see
that $C^*(X)$ and $C^*(B)$ are complete since $X$ and $B$ are simply
connected.  

To obtain a Symmetric Gorenstein Context we may now assume any one of
the three equivalent conditions (i) $X$ is Gorenstein, (ii) $F$ is Gorenstein or
(iii) $S\lra R$ is relatively Gorenstein.

To see they are equivalent note that (i) and (ii) are equivalent by  \cite{DGI}*{8.6}.
We have already noted that (iii) implies (i) in
Lemma~\ref{lem:relGorascdesc}. It remains to show that (i) implies
(iii). This follows from local duality as in
\cite[19.5]{JohnNotes}. Indeed, $C^*(B)$ is formal, so $C^*(B)\simeq
P$ where $P=k[x_1, x_2, \ldots, x_r]$ with $x_i$ of degree
$d_i<0$, and we may let $\fm=(x_1, x_2, \ldots, x_r)$ denote the
maximal ideal and the Gorenstein shift is $a_B=-d-r$ where
$d=d_1+\cdots +d_r$. Accordingly local duality for any small
$P$-module $M$ states that there is an equivalence
$$\Hom (M,P)\simeq \Sigma^{d+r}(\Gamma_{\fm}M)^{\vee},  $$
where $\Gamma_{\fm}$ is local cohomology at $\fm$ and $(-)^\vee$ is the $k$-dual. 
If $M$ is an $R$-module viewed as a $P$-module via a ring map $P\lra
R$ with $R$ small over $P$ then the equivalence may be taken to be one
of $R$-modules by taking a model of $R$ which is $P$-free.  
Now take $M=R=C^*(X)$; by (i) this is Gorenstein, of shift $a_X$ say.
Since $X$ is simply connected $C^*(X)$ automatically enjoys Gorenstein
duality (since $C_*(\Omega X)$ is connected and therefore has a unique action on $k$), so that
$$\Gamma_{\fm}(C^*(X))\simeq \Sigma^{a_X}C_*(X). $$
Hence
$$\Hom_{C^*(B)}(C^*(X), C^*(B))\simeq \Sigma^{a_X-a_B}C^*(B) $$
as required. 
%since then $R$ is Gorenstein by
%Gorenstein Ascent \cite{DGI2}*{Proposition~8.6}. Thus $S\lra R\lra Q$ comes from a
%fibration 
%$$B\lla E\lla F$$
%We note that 
%$$X\simeq E\Omega B \times_{\Omega B} F, \;\;
%C^*(X)=k\tensor_{C_*(\Omega B)}C^*(F)$$
%which is to say that $C^*(X)$ is the  $\Omega B$-Borel cochains of $F$.  
%$$\Hom_S(R,S)=\Hom_{C^*(B)}(C^*(X), C^*(B))\simeq \Hom_{C_*(\Omega
%  B)}(C_*(F), k)$$
%which is to say the $\Omega B$-Borel cochains on $DF$. 
%Since $B$ is formal, the equivalence $DF\simeq \Sigma^{-f} F$ may 
%be taken to be  $\Omega B$-equivariant and 
%$$\Hom_S(R,S)\simeq \Sigma^{-f}R, $$
%and the map is relatively Gorenstein. 

The final conclusion is that if $X$ is any Gorenstein space, we can
construct a normalization giving a  Symmetric Gorenstein Context. 

\subsection{An example from compact Lie groups}\label{ssec:SGCBG}
Once again we take $R=C^*(BG)$ and we suppose $G$ is a subgroup of 
a connected compact Lie group $U$ (for example by taking a faithful
represenation of $G$ in $U(n)$ and $U=U(n)$). We also assume that the
adjoint representation of $G$ is orientable over $k$ (for example if
$G$ is finite or connected or if $k$ is of characteristic 2). 

This gives the fibration 
$$ BU\lla BG \lla U/G$$
and the cofibration 
$$ C^*(BU)\lra C^*(BG)\lra C^*(U/G)$$
of algebras since connectedness of $U$ means the Eilenberg-Moore spectral sequence
converges.

Accordingly, we take $S=C^*(BU)$, $R=C^*(BG)$. This gives a Symmetric
Gorenstein Context. First,  we  find
$Q=C^*(U/G)$. Since $U$ is connected, $S$ is regular and since $Q$ is
finite, $R_*$ is finitely generated over $S_*$. If $S_*$ is regular,
it follows that $R$ is small over $S$.
%(later we will prove
%independence of normalization (Proposition \ref{prop:invariance}), from which it follows  that 
%$R$ is small over $S$ whenever $S$ is regular).  
Thus we again have $1+1$
g-regular, and that $S$ is Gorenstein. 

Finally 
$$\Hom_{C^*(BU)}(C^*(BG), C^*(BU))\simeq C^*(BG^{-L})$$
where $L$ is the tangent representation at $eG$ in $U/G$, see
\cite{BG7}*{Theorem~6.8} (with the proof completed in \cite{bgen}). Since $U$ is
connected and $ad(G)$ is orientable, $L$ is orientable and $S\lra R$ is relatively
Gorenstein. Thus, we  conclude $\mcF\simeq C_*(U)$ is Gorenstein. 

%If $G$  is also connected then  $\cE \simeq C_*(G)$,  and if $U/G$ is simply
%connected
%$\mcD \simeq C_*(\Omega U/G)$. 

%If we just assume that $U$ is connected and $k$ is of characteristic
%$p$
%then $\mcE =C_*(\Omega (BG_p^{\wedge}))$
%and  $\mcD =C_*(\Omega (U/G_p^{\wedge}))$. 

%If $k$ is of characteristic $p$ we could weaken this and just require that
%$\pi_0(U)$ is $p$-nilpotent. 

\section{Completions}
\label{sec:completion}
The notion of completeness occurs very naturally when passing between derived endomorphism algebras. Thus, unsurprisingly, it will play a key role in formulating a precise relationship between $R$ and $\cE$. As a quick reminder we recall the context from \cites{DG, DGI}. 

\subsection{Cellularization and completion}\label{ssec:completion}
We have already used the functor $EM=\Hom_R(k,M)$ from right $R$-modules to
right $\cE$-modules. Naturally $k$ is a left $\cE$-module, so $E$ has a
left adjoint $TX:=X\tensor_{\cE}k$. The counit of the adjunction 
$$TEM=\Hom_R(k,M)\tensor_{\cE}k \lra M$$
is evaluation and, provided $k$ is proxy-small, this is also the
$k$-cellularization \cites{DG, DGI}. 

Writing $\kdr =ER=\Hom_R(k,R)$, we have $TER=\kdr \tensor_{\cE} k$ and
the associated completion functor is 
$$\Lambda_k M:=\Hom_R(TER, M) \simeq\Hom_{\cE}(ER, EM)=\Hom_{\cE}(\kdr,
EM).$$
This has a universal property on $R$-modules, and in the setting of
classical commutative rings, the homotopy groups of $\Lambda_k M$ are 
given by the left derived functors of completion at the augmentation
ideal \cite{GM}*{Theorem~2.5}. 

We take from this the importance of the functor $\Eb$ defined by
$$\Eb X=\Hom_{\cE}(\kdr, X), $$
which  is naturally a module over
$$\Hom_\cE(\kdr, \kdr)\simeq \Hom_R(TER, R)=\Lambda_k(R), $$
the \emph{completion} of $R$. In this notation 
$$\Eb E M=\Hom_{\cE}(ER, EM)\simeq \Hom_R(TER, M)=\Lambda_k M, $$
the completion of $M$.

If $R$ is Gorenstein then $\kdr$ is a shift of $k$, and so
\begin{displaymath}
\Lambda_k(R) \simeq \Hom_\cE(k,k).
\end{displaymath}
Thus, if $R$ is Gorenstein and complete, $R$ and $\cE$ play interchangeable roles.

\subsection{The six Morita functors}\label{sec:6functors}
We apply the discussion of the previous section to all three rings $S,
R, Q$, using alphabetical mass-production. For the record, this gives functors
$$D\colon \Qmod \lra \modD,\;  E\colon \Rmod \lra \modE \;\,\mbox{and}\;\,  F\colon \Smod \lra \modF $$ 
  defined by 
$$D(L)=\Hom_Q(k, L) ,\; E(M)=\Hom_R(k, M) \;\mbox{ and }\; F(N)=\Hom_S(k, N). $$
These three functors are right adjoints; their left adjoints are given
by suitable tensor products with the left module $k$, but we will not introduce
special notation for these functors. 

For brevity, we write 
$$\Qh =\HomD(\kdq ,\kdq),\; \Rh =\HomE(\kdr,
\kdr)  \;\mbox{ and }\; \Sh =\HomF(\kdr, \kdr)$$
for the completions  of $Q, R$ and $S$, so that we
have maps
$$Q\lra \Qh, \; R\lra \Rh \text{ and } S\lra \Sh. $$
We then define functors 
$$\Db\colon \modD \lra \Qhmod , \Eb\colon \modE \lra \Rhmod 
\mbox{ and } \Fb\colon \modF \lra \Shmod $$ 
 by 
$$\Db (W)=\Hom_{\mcD}(\kdq, W),  \Eb (X)=\Hom_{\cE}(\kdr, X)\, 
\mbox{and}\;\, \Fb (Y)=\Hom_{\cF}(\kds, Y). $$
Again, these three  functors are right adjoints, but we will not need
to discuss their adjoint partners. 

\begin{rem}
\label{rem:comp}
When $R$ is small over $S$, as we always assume, the completion of an $R$-module
agrees with its completion as an $S$-module (or more precisely
the natural map gives an isomorphism $\Lambda^S  q^*M\simeq q^*\Lambda^R M$). 
Accordingly, we will simplify the notation and use $\Lambda $ in both cases. 
\end{rem}

\subsection{Finite generation is independent of complete Gorenstein normalization}
\label{subsec:Gorfg}

We show in this section that finite generation is independent of the chosen 
Symmetric Gorenstein Context provided our rings are complete. This considerably extends the results of Section
\ref{sec:dependence} in our main case of interest. 

\begin{prop}\label{prop:invariance}
Suppose we are given  $q\colon S\lra R$ with $R$ small over $S$, and both $R$ and $S$
g-regular and complete.  Provided $S\lra R$ is relatively Gorenstein, an $R$-module 
$M$ is $R$-small if and only if $q^*M$ is  $S$-small. 
\end{prop}

\begin{proof}
We have assumed $R$ is g-regular. Thus 
$k \finbuilds \cE$. Since $k$ is an $\cF$-module, $\cF \builds k$ over
$\cF$, and hence over $\cE$ by restriction. Hence  $j^*\cF\builds
\cE$ so, since $\cF$ and $\cE$ are small over $\cE$, we see $j^*\cF \finbuilds \cE$ by 
Thomason's Localisation Theorem \cite{NeeGrot}*{Theorem~2.1}. 
 
Now consider an $R$-module $M$. Since $S\finbuilds R$, it is clear that if  $R \finbuilds M$
(over $R$ and hence over $S$ by restriction)
then $S\finbuilds q^*M$.

On the other hand, suppose   $S \finbuilds q^*M$. We then see that as 
$\cF$-modules 
$$FS \finbuilds Fq^*M\simeq j_*EM=\cF\tensor_{\cE}EM,$$
where the first equality is via Lemma~\ref{lem:topface} below.
This then remains true after applying $j^*$, and since $j^*\cF
\finbuilds \cE$, 
$$j^*\cF \tensor_{\cE} EM  \finbuilds \cE \tensor_{\cE}EM\simeq EM.$$

In fact 
$$j^*FS= j^*\Hom_S(k,S)\simeq \Hom_R(k, \Hom_S(R,S))=E\Hom_S(R,S).$$
Thus
$$E\Hom_S(R,S)\finbuilds EM, $$
and we may apply $\Eb$ to see 
$$\Lambda \Hom_S(R,S)=\Eb E \Hom_S(R,S) \finbuilds \Eb EM=\Lambda M, $$
so that in the relatively Gorenstein case,  the completion of $R$
finitely builds the completion of $M$. Since $S$ is complete by
hypothesis and $q^*M$ is small, $q^*M$ is complete over $S$  and hence
$M$ is complete over $R$ which is, by assumption, itself complete. Thus $R$ builds $M$ as claimed.
\end{proof}

\begin{cor}
\label{cor:Gorfg}
If $R$ is complete, any relatively Gorenstein normalization $q\colon S\lra R$, such that $S$ is complete, 
defines the same notion of finite generation.  
\end{cor}

\begin{proof}
Suppose we have two such relatively Gorenstein normalizations
$S_1\lra R$ and $S_2 \lra R$. We have a commutative diagram
$$\diagram 
S_1\times S_2 \dto \rto & S_1 \dto \\
S_2 \rto &R
\enddiagram $$
of ring spectra. Given an $R$-module $M$, this is small over $S_1$ if and only if it is
small over $S_1\times S_2$ by Proposition \ref{prop:invariance}, and
similarly it is small over $S_2$ if and only if it is small over
$S_1\times S_2$. Accordingly it is small over $S_1$ if and only if it
is small over $S_2$ as required. 
\end{proof}

This permits us to understand small objects over $g$-regular rings in
considerable generality. 

\begin{cor}\label{cor:c-small}
Let $S$ be a complete and g-regular augmented ring spectrum. Suppose that $S$ admits a relatively Gorenstein normalization $T\lra S$ such that $T$ is complete and coefficient-regular. Then an $S$-module $N$ is small if and only if $N_*$ is finitely generated over $S_*$.  

Accordingly, if $q\colon S\lra R$ is normalization of a ring spectrum $R$ then 
$$\sfD^{q\mathrm{-b}}(R)=\sfD^f(R)=\{ M \st M_* \mbox{ is finitely generated
  over } R_*\}.$$
\end{cor}

\begin{proof}
%It is clear that if $N$ is small then $N_*$ is finitely generated over $S_*$.
Let us choose $T\lra S$ a complete coefficient-regular normalization as in the statement, and recall that by definition (see \ref{subsec:cregsmall}) $T_*$ is noetherian. Since $S$ is small over $T$ it follows that $S_*$ is a finitely generated $T_*$-module and hence $S_*$ is itself noetherian. Thus if $N$ is small over $S$ the homotopy $N_*$ is finitely generated over $S_*$.

On the other hand we suppose $N_*$ is finitely
generated over $S_*$. Then
since $S_*$ is finitely generated over $T_*$, the module  
$N_*$ is finitely generated over $T_*$, and as observed in Subsection \ref{subsec:cregsmall}
$N$ is small over $T$. By Proposition \ref{prop:invariance} the $S$-module $N$
is also small. 
\end{proof}

It is worth making one special case explicit. 

\begin{ex}\label{ex:BG-c-small}
If $G$ is a finite $p$-group then a $C^*(BG)$-module $M$ is small if
and only if $M_*$ is finitely generated over $H^*(BG)$. This follows from Corollary~\ref{cor:c-small} applied to the normalization discussed in Section~\ref{ssec:SGCBG}, i.e. the map $C^*(BU(n))\lra C^*(BG)$ induced by a faithful representation $G\lra U(n)$. 
\end{ex}

\section{Commutation relations}\label{sec:cr}

Assuming a  Symmetric Gorenstein Context, as in Definition~\ref{defn:sgc} whose notation we follow, we have  defined, in Section~\ref{sec:6functors}, six functors $D,E,F$ 
$\Db, \Eb, \Fb$ relating a number of module categories. These satisfy
a large number of commutation relations, that we describe in this section. As these commutativity relations might be 
of interest in more general situations we are precise about exactly what is used at each step.

\begin{thm}\label{thm:relations}
Given a Symmetric Gorenstein Context, we have eight commutation relations between our functors, summarized
by the fact that the eight squares in the following diagrams commute. 
$$\diagram 
\Smod \rto^{F}&\modF \rto^{\Fb}&\Shmod\\
\Rmod \rto^{E}\uto^{q^*}&\modE \rto^{\Eb}\uto^{j_*}&\Rhmod \uto^{\qh^*}\\
\Qmod \rto^{D}\uto^{p^*}&\modD \rto^{\Db}\uto^{i_*}&\Qhmod \uto^{\ph^*}\\
\enddiagram$$

$$\diagram 
\Smod \rto^{\Sigma^{-a_S}F}\dto^{q_*} &\modF \rto^{\Sigma^{a_S}\Fb}\dto^{j^*} &\Shmod\dto^{\qh_*} \\
\Rmod \rto^{\Sigma^{-a_R}E}\dto^{p_*}&\modE \rto^{\Sigma^{a_R}\Eb}\dto^{i^*}&\Rhmod \dto^{\ph_*}\\
\Qmod \rto^{\Sigma^{-a_Q} D}&\modD \rto^{\Sigma^{a_Q}\Db}&\Qhmod\\
\enddiagram$$

We recall that $\Rh$ is the completion of $R$, and so on for the other ring spectra, and $\qh\colon \Sh \lra \Rh$ is the map induced on completions by $q\colon S\lra R$, and so on for the other maps.
\end{thm}

\begin{rem}
We note that there are no suspensions in the top diagram, and that in
the lower diagram each of the functors has a shift equal to
plus or minus the Gorenstein shift of the two rings in the relevant
row. For instance, this corresponds to the fact that we have $a_q = a_S - a_R$.
\end{rem}

The strategy of proof is to prove that the upper two squares in the
first and second diagram commute. The commutation of the lower two
will then follow by using the symmetry of the Symmetric Gorenstein
Context. 

The arguments for commutation of the two squares are similar for 
the first and second diagrams, but in view of the suspensions, some
differences are inevitable.

\subsection{The  diagram without suspensions}
We will show that the top two squares in the top diagram commute
(i.e., those involving  $q^*$ and $j_*$ and the 
Morita functors). 

%$$\diagram 
%&\Smod \ar[rrr]^F \ar[dddrrr]_{\Lambda}&&&\modF\ar@{->}[ddd]^{\Fb}\\
%\Rmod \ar[rrr]^E \ar[ru]^{q^*}\ar[dddrrr]_{\Lambda}&&&\modE\ar@{->}[ddd]^{\Eb} \ar[ru]^{j_*}&\\
%&&&&\\
%&&&&\Shhmod\\
%&&&\Rhhmod \ar[ru]^{\qh^*}&
%\enddiagram$$
We remark that the two horizontal composites are completion  by 
the discussion in Section \ref{sec:completion}, and by Remark
\ref{rem:comp} the two completions are compatible under restriction, i.e.\ 
the outer rectangle commutes.

\subsection{The top left hand square}

\begin{lem}
\label{lem:topface}
The top left hand square  commutes in the sense that for any $R$-module $M$ we have
a natural equivalence
$$j_*EM \stackrel{\simeq}\lra F(q^*M). $$
\end{lem}

\begin{proof}
We have 
$$j_*EM =\Hom_R(k, M)\tensor_{\cE} \Hom_S(k,k)\simeq 
\Hom_R(k, M)\tensor_{\cE} \Hom_R(R\tensor_S k,k), $$
and there is a natural evaluation map to 
$$\Hom_R(R\tensor_S k, M). $$

Indeed, we have a map 
$$\Hom_R(k, M)\tensor_{\cE} \Hom_R(T,k) \lra \Hom_R(T, M)$$
for any $R$-module $T$. It is evidently an isomorphism when $T=k$ and
hence for any module finitely built from $k$. In particular this
applies to $T=R\tensor_Sk $, which is finitely built by $k$ as in the proof of 
Lemma~\ref{lem:proxy}, to give an isomorphism
\begin{displaymath}
j_*EM \simeq \Hom_R(k, M)\tensor_{\cE} \Hom_R(R\tensor_S k,k) \simeq \Hom_R(R\tensor_S k, M).
\end{displaymath}
It then just remains to note that $\Hom_R(R\tensor_S k, M) \simeq \Hom_S(k,q^*M) = F(q^*M)$.
\end{proof}

\subsection{The top right square}
For the right hand square one needs to use a little more. Of course, the conditions we require 
 hold in the case of principal interest i.e.\ the Symmetric Gorenstein Context.

\begin{prop}
\label{prop:rhend}
Suppose $R$ and $S$ are Gorenstein and $S\lra R$ is relatively
Gorenstein.  
For an $\cE$-module $X$ there is a natural equivalence
$$\qh^*\Eb X\stackrel{\simeq}\lra \Fb j_*X. $$
\end{prop}

\begin{proof}
We begin by noting that if $\Hom_S(R,S)\simeq \Sigma^{a_q}R$ then 
$$\kdr \simeq  \Hom_{R}(k, \Sigma^{-a_q}\Hom_S(R,S))\simeq
\Sigma^{-a_q}\Hom_S(k, S)=\Sigma^{-a_q}\kds. $$
%$$\kdr=\Hom_{R}(k, R)\simeq  \Hom_{R}(k, \Sigma^{-a_q}\Hom_S(R,S))\simeq
%\Sigma^{-a_q}\Hom_S(k, S)=\Sigma^{-a_q}\kds. $$
Thus in particular, the $\cE$-module $\kdr$ is the restriction of the
$\cF$-module $\Sigma^{-a_q}\kds$.

We have
$$\Fb j_*X=\Hom_{\cF}(\kds, X\tensor_{\cE}\cF) , $$
and 
$$\qh^*\Eb X=\qh^*\Hom_{\cE}(\kdr, X)\simeq
\Hom_{\cF}(\Sigma^{-a_q} \kds, \Hom_{\cE}(\cF, X)). $$

Now, we have a natural equivalence
$$\Sigma^{a_q} X\tensor_{\cE}\Hom_{\cE}(\cF, \cE)
\stackrel{\simeq}\lra
\Sigma^{a_q}\Hom_{\cE}(\cF, X),$$
where the equivalence uses the fact (Lemma \ref{lem:MoritaRsmall})
that $\cF$ is small over $\cE$. Finally,  $\Sigma^{a_q}\Hom_{\cE}(\cF, \cE) \simeq \cF$, since $j$ is
relatively Gorenstein and the claimed identification follows. 
\end{proof}

\subsection{The  diagram with suspensions}

The first row of the second diagram relates $q_*$ and $j^*$, and by contrast with the
first, this one involves suspensions. 

%We consider the  the diagram
%$$\diagram 
%&\modF \ar[rrr]^{\Fb} \ar[dl]^{j^*}\ar[dddrrr]&&&\Shmod \ar[dl]^{\hat{q}_*}\ar@{->}[ddd]^{F}\\
%\modE \ar[rrr]^{\Eb} \ar[dddrrr]&&&\Rhmod \ar@{->}[ddd]^{E}&\\
%&&&&\\
%&&&&\modF \ar[dl]^{j^*} \\
%&&&\modE &
%\enddiagram$$

The functors $E$ and $F$ include implicit restrictions $\Rhmod \lra
\Rmod$, $\Shmod \lra \Smod$, which are the identity if we assume $R$
and $S$ are complete. 

We first deal with the composites.

\begin{lem}
We have a natural isomorphism
$$F\Fb Y\simeq Y$$
for $\cF$-modules $Y$ and a natural equivalence
$$E\Eb X\simeq \Hom_{\cE}(\kdr \tensor_R k, X)$$
for $\cE$-modules $X$. When $\kdr \simeq \kde$ this is completion. 
\end{lem}

\begin{proof}
We calculate directly that
$$E\Eb X =\Hom_R(k, \Hom_{\cE}(\kdr, X))\simeq \Hom_{\cE}(\kdr
\tensor_R k, X), $$
and similarly for $F\Fb Y$. We note that there is
always a natural map 
$$\kdr \tensor_Rk =\Hom_R(k, R)\tensor_Rk \lra \Hom_R(k,k)=\cE , $$
but we only know it is an equivalence if $k$ is small over $R$. Since
$S$ is g-regular, the corresponding map is an equivalence for $S$ which shows $F\Fb Y \simeq Y$.
\end{proof}

%The two triangular panels at the front and back commute by Lemma
%\ref{lem:CEisLambda}. The bottom face commutes by Remark
%\ref{rem:comp} (i). This leaves the top face and the right hand end. 
\subsection{The top left square}
The next relation is straightforward.

\begin{lem}
\label{lem:fourthcomm}
Assume that $R$ and $S$ are Gorenstein and $S\lra R$ is relatively
Gorenstein of shift $a_q$.

For any $S$-module $N$ we have
a natural equivalence
$$j^*F N \simeq \Sigma^{a_q}Eq_*N. $$
\end{lem}

\begin{proof}
We have 
$$Eq_*N=\HomR (k,R\tensor_S N). $$
On the other hand 
$$j^*FN =j^*\HomS(k,N)\simeq \HomR(k, \HomS(R,N)). $$
The relation then follows since $R$ is small over $S$, so that 
$$\HomS(R,N)\simeq \HomS(R,S)\tensor_SN\simeq
\Sigma^{a_q}R\tensor_S N.$$
\end{proof}

\subsection{The top right square}
The final square is a little trickier. 

\begin{lem}
\label{lem:thirdcomm}
Assume that $R$ and $S$ are Gorenstein and $S\lra R$ is relatively
Gorenstein of shift $a_q$.

For any $\cF$-module $Y$ we have
a natural equivalence
$$\qh_*\Fb Y \stackrel{\simeq}\lra \Sigma^{-a_q}\Eb (j^*Y). $$
\end{lem}

\begin{proof}
First, we note that since $q$ is relatively Gorenstein, $j^*\kds
\simeq \Sigma^{a_q}\kdr$:
$$\kds =\Hom_S(k,S)\simeq \Hom_R(k , \Hom_S(R,S))\simeq
\Sigma^{a_q}\kdr .$$
In particular
$$\Rh  =\Hom_{\cE}(\kdr, \kdr) \simeq \Hom_{\cE}(j^*\kds, j^*\kds).$$

Thus, we find
\begin{align*}
\qh_*\Fb Y &=\Hom_{\cE} (j^*\kds, j^*\kds)\tensor_{\Sh}
\Hom_{\cF}(\kds, Y) \\ 
&\simeq \Hom_{\cF} (j^*\kds \tensor_{\cE} \cF, \kds)\tensor_{\Sh}
\Hom_{\cF}(\kds, Y).
\end{align*}
There is a natural evaluation map to 
$$\Hom_{\cF}(j^*\kds \tensor_{\cE}\cF, Y)\simeq
\Hom_{\cE}(\Sigma^{a_q}\kdr , j^*Y)\simeq \Sigma^{-a_q}\Eb j^*Y.$$

As in the proof of Lemma~\ref{lem:topface} it suffices to show that $\kds \finbuilds \kdr\tensor_{\cE}\cF$. 
Since $\kds =FS$ it suffices to show that $\kdr \tensor_{\cE} \cF$ is
the image of a small $S$-module under $F$, and in fact we show it is $F(R)$.

For this (recalling from Lemma~\ref{lem:proxy} that $k\finbuilds R\tensor_Sk$ for the third
equivalence), we compute that 
$$\begin{array}{rcl}
\kdr \tensor_{\cE}\cF 
&=& \Hom_R(k,R)\tensor_{\cE} \Hom_S(k,k) \\
&\simeq & \Hom_R(k,R)\tensor_{\cE} \Hom_R(R\tensor_Sk,k) \\
&\simeq & \Hom_R(R\tensor_Sk,R)\\
&\simeq & \Hom_S(k,R)\\
&=& F(R).
\end{array}$$
\end{proof}

\subsection{The symmetric counterparts}
\label{subsec:commutationsummary}
We have so far shown that the top two squares in the two diagrams
commute. In other words, we have  established
four relations: 

\begin{align*}
&Fq^*M \simeq j_*E M \\
&\Fb j_*Y \simeq \qh^*\Eb Y\\
&\Eb j^*Z \simeq \Sigma^{a_{S}-a_{R}}\qh_*\Fb Z \\
&j^*FN\simeq \Sigma^{a_S-a_R}Eq_* N
\end{align*}
In the symmetric context we obtain some more by replacing $S\lra R$ by
$\mcD \lra \mcE$ (and hence  $\mcF \lla \mcE$  by $Q\lla R$). 

In giving the symmetric relations, we need to bear in mind that
$EM=\Hom_R(k,M)$ corresponds to 
$$\Eb'Y=\Hom_{\cE}(k, Y)\simeq
\Sigma^{a_R}\Hom_{\cE}(\kdr , Y)=\Sigma^{a_R}\Eb Y$$
and 
$\Eb Y=\Hom_{\cE} (\kdr ,Y)$ corresponds to 
$$E' M=\Hom_{R}(\kde, M)\simeq
\Sigma^{-a_{\cE}}\Hom_{R}(k , M)=\Sigma^{-a_{\cE}}EM. $$
This allows us to establish the commutation of the lower two squares in the two
diagrams,  expressed as  equations in the following lemma.

\begin{lem}\label{lem_isoms}
In a Symmetric Gorenstein Context,  there are natural isomorphisms for $X\in \Modu \mcE$, $Y\in \Modu Q$, $M\in \Modu R$, and $N\in \Modu \mcD$
\begin{align*}
&\Sigma^{a_{\mcD}}\Db i^*X \simeq \Sigma^{a_{\mcE}}\ph_*\Eb X \\
&\Sigma^{-a_{\mcD}}Dp_*M \simeq \Sigma^{-a_{\mcE}}i^*EM \\
&Ep^*L \simeq i_*D L \\
&\ph^*\Db N\simeq \Eb i_* N
\end{align*}
\end{lem}

\begin{proof}
Applying Lemma \ref{lem:topface}, Proposition
\ref{prop:rhend},  Lemma \ref{lem:thirdcomm} and  Lemma \ref{lem:fourthcomm} 
to the Morita
counterparts, we obtain
\begin{align*}
&\Db' i^*X \simeq \ph_*\Eb' X \\
&D'p_*M \simeq i^*E'M \\
&E'p^*L \simeq \Sigma^{a_{\mcD}-a_{\mcE}}i_*D' L \\
&\ph^*\Db' N\simeq \Sigma^{a_{\mcD}-a_{\mcE}}\Eb' i_* N
\end{align*}
Inserting appropriate suspensions, recalling that Morita counterparts
have the same shift (i.e., $a_R=a_{\cE}$ etc), and that Gorenstein
ascent gives $a_R=a_S+a_Q$, we obtain the stated results. 
\end{proof}

\section{Morita equivalences and singularity categories}\label{sec:duality}

We have now introduced all the apparatus necessary to prove our main
result, which gives an equivalence of the bounded derived categories of Morita 
counterparts occuring in a Symmetric Gorenstein Context. As a consequence 
we can describe how singularity categories behave under
Morita equivalence (or Koszul duality if the reader prefers). 

\subsection{An equivalence of bounded derived categories}

Let us suppose we are given a Symmetric Gorenstein Context (see Definition~\ref{defn:sgc}, and see Section~\ref{sec:6functors} for the relevant functors) consisting of cofibre sequences
\begin{displaymath}
S \stackrel{q}{\lra} R \stackrel{p}{\lra} Q
\end{displaymath}
and
\begin{displaymath}
\cF \stackrel{j}{\lla} \cE \stackrel{i}{\lla} \mcD
\end{displaymath}
where $R,S,\cE,$ and $\mcD$ are assumed complete. We have defined analogues of the bounded derived category for $R$ and $\cE$, namely
\begin{displaymath}
\sfD^{q\mathrm{-b}}(R)=\DRfbbS \quad \text{and} \quad \sfD^{i\mathrm{-b}}(\cE) = \DEfbbD
\end{displaymath}
and seen in Corollary~\ref{cor:Gorfg} that in fact under mild
hypotheses (see Proposition~\ref{prop:invariance})  these subcategories
do not depend on the chosen normalizations. 

In this section we prove our main theorem:

\begin{thm}\label{thm_realmain}
Suppose we are given a Symmetric Gorenstein Context as above with $R,
S, \mcD$ and $\cE$ complete. Then
\begin{displaymath}
E = \Hom_R(k,-)\colon \Rmod \lra \modE
\end{displaymath}
and
\begin{displaymath}
 \Eb = \Hom_\cE(\kdr,-)\colon \modE \to \Rmod
\end{displaymath}
restrict to quasi-inverse equivalences
\begin{displaymath}
\sfD^{q\mathrm{-b}}(R)=\DRfbbS \simeq \DEfbbD =\sfD^{i-b}(\cE).
\end{displaymath}
\end{thm}

The first matter of business is to check that $E$ and $\Eb$ both restrict to functors between the bounded derived categories. We will state the necessary lemmas for both cofibre sequences, but we will only prove them for the one involving $S, R$, and $Q$; in all cases the proofs are, \emph{mutatis mutandis}, the same.

\begin{lem}\label{lem_finiteness1}
Let $M$ be an $R$-module such that $q^*M$ is small over $S$. Then $p_*M$ is finitely built by $k$. Similarly if $X$ is an $\mcE$-module such that $i^*X$ is small over $\mcD$, then $j_*X$ is finitely built by $k$.
\end{lem}
\begin{proof}
Suppose $M$ is as given. Then we have
\begin{displaymath}
(S \vDash q^*M) \Rightarrow (k \simeq k\otimes_S S \vDash k\otimes_S q^*M \simeq Q\otimes_R M = p_*M).
\end{displaymath}
\end{proof}

\begin{lem}\label{lem_finiteness2}
Let $M$ be an $R$-module such that $p_*M$ is finitely built by $k$. Then $i^*EM$ is small over $\mcD$. Similarly, if $X$ is an $\mcE$-module such that $j_*X$ is finitely built by $k$ then $q^*\Eb X $ is small over $S$. 
\end{lem}
\begin{proof}
Let $M$ be as in the statement. Then we have
\begin{displaymath}
(k\vDash p_*M) \Rightarrow (\mcD = Dk \vDash Dp_*M \simeq i^*EM)
\end{displaymath}
(up to suspensions which are irrelevant for statements about building), where the last isomorphism above is via Theorem~\ref{thm:relations}.
\end{proof}

Thus $E$ and $\Eb$ restrict to functors
\begin{displaymath}
\xymatrix{
\DRfbbS \ar[r]<0.5ex>^-E \ar@{<-}[r]<-0.5ex>_-{\Eb} & \DEfbbD
}.
\end{displaymath}

It just remains to check they are inverse to one another on these categories.

\begin{proof}[Proof of Theorem~\ref{thm_realmain}]
%As in the proof of \cite{GreenleesSing}*{Theorem~5.1} one sees
%$\EbE$ is the identity on $R$-modules $M$ with $q^*M$ small
%over $S$. 
Since $R$ is complete, the composite $\Eb E $ is the identity on $R$-modules $M$
with $q^*M$ small over $S$. Indeed, if 
$$\left( S\finbuilds q^*M\right) \Rightarrow \left( \Lambda S=\Fb F S \finbuilds
 \Fb F q^* M\simeq q^* \Eb E M\right) .$$ 
Since $S$ is complete $\Fb F S=\Lambda S\simeq  S$, so the above yields that $q^* \Eb E M$ is finitely built
by $S$. Completeness of $S$ also tells us that $\Fb F$ is equivalent to the identity on small $S$-modules. It follows that if we apply $q^*$ to the completion $M \lra
\Eb EM =\Lambda M$ then it is an equivalence. However $q^*$ reflects isomorphisms so $M\simeq \Eb EM$ as required. 

On the other hand suppose $X$ is an $\mcE$-module with $i^*X$ small over $\mcD$. In $\Modu \mcD$ we have
\begin{align*}
D\Db(\mcD) &= D(\Hom_\mcD(k^{\#Q}, \mcD)) \\
&\simeq \S^{-a_Q}D(k^{\#\mcD}) \\
&\simeq \S^{-a_Q + a_\mcD}D(k) \\
&\simeq \mcD
\end{align*}
where we have used $a_Q = a_\mcD$. Thus $D\Db$ is the
identity on objects finitely built by $\mcD$. By the analogue of Remark \ref{rem:comp}(i) or using the relations from Theorem~\ref{thm:relations} we see that restriction and completion commute for $X$ and so
\begin{displaymath}
i^*E\Eb X \to i^*X
\end{displaymath}
is an isomorphism. Since $i^*$ reflects isomorphisms this shows $E\Eb X \to X$ is already an isomorphism. Thus $E\Eb$ is isomorphic to the identity on $\DEfbbD$ and so we have the claimed equivalence
\begin{displaymath}
\DRfbbS \simeq \DEfbbD.
\end{displaymath}
\end{proof}

\subsection{Singularity and cosingularity categories}
Let us now formally introduce singularity and cosingularity categories and record the consequence of our theorem for their behaviour under Morita equivalence. 

The singularity category of an ordinary ring $R$  is designed to measure how
far $R$ is from being regular. It is defined as the
Verdier quotient of the bounded derived category by the complexes finitely built by $R$:
$$\sfD_\mathrm{sg}(R):= 
\frac{\sfD^\mathrm{b}(R)}{\sfD^\mathrm{c}(R)}. $$

\begin{defn}\label{defn:qsg}
Accordingly, for a potentially more exotic ring $R$ together with a normalization $S\stackrel{q}{\lra} R$, we define
$$\sfD_{q\mathrm{-sg}}(R):=
\frac{\sfD^{q\mathrm{-b}}(R)}{\sfD^\mathrm{c}(R)}=
\frac{\DRfbbS}{\DRfbbR}.$$
\end{defn}
Again this provides a measure of how far $R$ is from being g-regular, although this is made more subtle by the involvement of normalizations.

\begin{lem}\label{lem:singreg}
If there exists a normalization $S\stackrel{q}{\lra}R$ such that we have $\sfD_{q\mathrm{-sg}}(R) \simeq 0$ then $R$ is g-regular. On the other hand, if $R$ is g-regular and complete then for every relatively Gorenstein normalization $S\stackrel{q}{\lra} R$, such that $S$ is complete, we have $\sfD_{q\mathrm{-sg}}(R)\simeq 0$.
\end{lem}
\begin{proof}
First suppose there exists an $S\stackrel{q}{\lra}R$ such that $\sfD_{q\mathrm{-sg}}(R) \simeq 0$. Then, since $k$ is small over $S$, it certainly lies in $\DRfbbS$ and thus must be killed upon the passage to the singularity category. This says precisely that $k$ is small over $R$ i.e.\ $R$ is regular.

The second statement is a direct consequence of Proposition~\ref{prop:invariance}.
\end{proof}

Given that we work with augmented ring spectra it is 
natural to introduce the dual notion. 

\begin{defn}\label{defn:coregular}
We say $R$ is \emph{coregular} if it is finitely built from $k$ in the sense that
\begin{displaymath}
R\in \thick(k) \subseteq \sfD(R). 
\end{displaymath}
\end{defn}

We then define the cosingularity category to measure how far $R$
is from being coregular.

\begin{defn}\label{defn:qcsg}
The \emph{cosingularity category} of $R$ with respect to the normalization $q$ is
$$\sfD_{q\mathrm{-cosg}}(R):=
\frac{\sfD^{q\mathrm{-b}}(R)} 
{\DRfbbk}=
\frac{\DRfbbS}{\DRfbbk}. $$
\end{defn}
Again this idea of measuring can be made somewhat precise.

\begin{lem}
If there exists a normalization $S\stackrel{q}{\lra}R$ such that we
have $\sfD_{q\mathrm{-cosg}}(R) \simeq 0$ then $R$ is coregular.
\end{lem}
\begin{proof}
If the cosingularity category vanishes then, since $R$ is an object of $\DRfbbS$, we see $k\finbuilds R$ i.e.\ $R$ is coregular.
\end{proof}

\begin{rem}
Inspired by noncommutative algebraic geometry, the cosingularity category could also be viewed as an analogue of 
the bounded derived category of coherent sheaves on the ``projective scheme'' 
associated to $R$, i.e.\ we might think in terms of an equation $\sfD^\mathrm{b} (\Proj(R)):=\sfD_\mathrm{cosg}(R)$. 
\end{rem}

Again, in view of Corollary \ref{cor:Gorfg}, amongst normalizations $q\colon S\lra
R$ giving a Symmetric Gorenstein Context with both rings complete,  these categories are both
independent of $q$, and we simply write $\Dsg (R)$,
$\Dcosg (R)$ in this case. 

\subsection{Morita functors and singularity categories}
As one might expect from Koszul duality, taking Morita counterparts switches the roles of the singularity and cosingularity categories. 

\begin{thm}\label{thm:main}
Suppose $R, S, \mcE$, and $\mcD$ are complete Gorenstein, $R$ and $k$
are small over $S$ and $S\stackrel{q}{\to}R$ is relatively
Gorenstein. Then $E$ and $\Eb$ induce equivalences
\begin{displaymath}
\sfD_{q\mathrm{-sg}}(R) = \frac{\DRfbbS}{\DRfbbR} \simeq \frac{\DEfbbD}{\DEfbbk}
= \sfD_{i\mathrm{-cosg}}(\mcE)
\end{displaymath}
and
\begin{displaymath}
\sfD_{q\mathrm{-cosg}}(R) = \frac{\DRfbbS}{\DRfbbk}\simeq \frac{\DEfbbD}{\DEfbbE}
= \sfD_{i\mathrm{-sg}}(\mcE).
\end{displaymath}
\end{thm}

\begin{proof}
Given the equivalence of Theorem~\ref{thm_realmain} this comes down to checking the thick subcategories we wish to take quotients by are identified. We first note that since $R$ and $k$ are small over $S$, and $\mcE$ and $k$ are small over $\mcD$, both expressions make sense. It then just remains to note that
\begin{align*}
E(R) &\simeq \S^{a_R}k & E(k) &\simeq \mcE \\
\Eb(k) &\simeq \S^{-a_R}R & \Eb(\mcE) &\simeq k.
\end{align*}
\end{proof}

\section{Examples}
\label{sec:egs2}

This section gives a number of examples illustrating the main theorem in the various contexts we have kept in mind throughout. First of all, we begin with the situation that $R$ is itself
regular. In that case we can take $S=R$ and so our Symmetric
Gorenstein context is $R\lra R\lra k$ and $\cE\lla \cE\lla k$. 
Of course, in this situation
\begin{displaymath}
\Dsg(R) \simeq  0 \simeq  \Dcosg(\cE).
\end{displaymath}
However, we do obtain non-trivial equivalences
$$\sfD^\mathrm{b}(R) \simeq  \sfD^\mathrm{b}(\cE) \;\; \text{and} \;\;
\Dcosg (R)\simeq \Dsg (\cE). $$
Despite the strong assumption on $R$ there are several important examples. 

\begin{ex}{\em (Koszul duality)}
\label{eg:DsgKoszul}
Returning to Example~\ref{ex:Koszul} we could take $R = \Lambda$ a Gorenstein Koszul algebra of finite global dimension viewed as a DG-algebra with trivial differential. In this case $\cE\simeq \Lambda^!$ is also formal and we recover Koszul duality in this setting:
\begin{displaymath}
\sfD^\mathrm{b}(\Lambda) \simeq  \sfD^\mathrm{b}(\Lambda^!) \;\;
\text{and} \;\; \Dcosg (\Lambda)\simeq \Dsg (\Lambda^!).
\end{displaymath}
There are many concrete examples: for instance we could take for
$R=k[x_0, \ldots , x_n]$ a graded polynomial ring and then get for
$\cE$ an exterior algebra $\Lambda (\tau_0,\ldots , \tau_n)$, as in
the classical BGG correspondence, or we could take $R=k\langle a_1, \ldots , a_n\rangle /(a_1^2+a_2^2+\cdots
+a_n^2)$ where $k\langle a_1, \ldots , a_n\rangle$ denotes the free algebra on the $a_i$ with the standard grading, which is also Koszul of finite global dimension, and find that $\cE$ is quasi-isomorphic to the graded ring $k[x_1, \ldots, x_n]/(x_ix_j, x_i^2-x_j^2\st i\neq j)$ viewed as a DG-algebra.
\end{ex}

%\begin{ex}
%If we take $R=k[x_0, \ldots , x_n]$ we obtain $\cE=\Lambda (\tau_0,
%\ldots , \tau_n)$ (when is this formal. Sufficient to have $d_i=\deg
%(x_i) \geq 0$   and the BGG correspondence
%$$\Dsg (\Lambda (\tau_0, \ldots, \tau_n))\simeq \Dcosg (k[x_0, \ldots
%, x_n])=D^b(\mathbb{P}(d_0, \ldots, d_n). $$
%\end{ex}

%\begin{ex}
%We may take $R=k\langle a_1, \ldots , a_n\rangle /(a_1^2+a_2^2+\cdots
%+a_n^2)$ and find $\cE =k[x_1, \ldots, x_n]/(x_ix_j, x_i^2-x_j^2\st
%i\neq j)$. This shows that $R$ is g-regular; it may be better to start
%with $\cE$ which is a Koszul ring and hence find that the Koszul dual
%$R$ is formal. In any case we find
%$$\Dcosg (R)\simeq \Dsg (\cE). $$

%This applies more generally: if $\cE $ is a finite dimensional Koszul
%ring then $R$ is formal and g-regular with  
%$$\Dcosg (R)\simeq \Dsg (\cE). $$
%\end{ex}

\begin{ex}{\em (Ginzburg DGAs)}
\label{eg:DsgGinzburg}
We could fix a quiver with potential $(Q,w)$ and take for $R$ the smooth DG-algebra $\Gamma(Q,w)$, known as the Ginzburg DGA. We refer to \cite{KellerDCY} for further details and the fact that $\Gamma(Q,w)$ is bimodule Calabi-Yau and hence Gorenstein. In this case, the cosingularity category of $\Gamma(Q,w)$ is called the (generalised) cluster category $\mathcal{C}_{(Q,w)}$ associated to our quiver with potential \cite{Amiot}*{Definition~3.5}. Theorem~\ref{thm:main}, slightly generalized by replacing $k$ by a semisimple ring, gives an alternative description of the generalized cluster category:
\begin{displaymath}
\mathcal{C}_{(Q,w)} \simeq \Dcosg(\Gamma(Q,w)) \simeq \Dsg(\cE).
\end{displaymath}
\end{ex}

\begin{ex}
We may take $R$ to be a complete discrete valuation ring with residue
field $\Fp$ and function field $K$. This gives $\cE$ with $\cE_*=\Lambda_{\Fp}(\tau_{-1})$
(as shown in \cite{DGI4} this gives all such $\cE$ up to
quasi-isomorphism). We then find
$$\Dsg (\cE) \simeq \Dcosg (R)=\frac{\sfD^\mathrm{b}(R)}{\Thick(\Fp)}
\simeq  \sfD^\mathrm{b}(K), $$
where $\thick(\Fp)=\sfD^\mathrm{b}(R\finbuiltby \Fp)$ can also be described as the full subcategory consisting of objects supported just at the maximal ideal of $R$.
\end{ex}

\begin{ex} {\em (Rational spaces)}\label{ex:rat11}
We may take $R=C^*(X;\Q), k=\Q$ for any Gorenstein rational space $X$
(in this context, $R$ being Gorenstein coincides with the definition  of \cite{FHT} ).
 As in Example \ref{ex:dt}, the Eilenberg-Moore theorem gives
 $\cE\simeq C_*(\Omega X; \Q)$. By Noether normalization, $H^*(X)$ is
 finite over a polynomial subring and we may choose a map $X\lra B$
 with $B$ a product of even Eilenberg-MacLane spaces with finite fibre
 $F$, and this gives a Symmetric
 Gorenstein Context. As $C^*(B)$ is coefficient-regular we know, from Proposition~\ref{prop:cfgisf}, that
\begin{align*}
\sfD^\mathrm{f}(C^*(X)) &= \{M\in \sfD(C^*(X))\; \vert \; H^*(M) \; \text{is finitely generated over}\; H^*(X)\} \\
&=\sfD(C^*(X)\finbuiltby C^*(B))
\end{align*}
We then find
$$\Dsg (C^*(X)) =\frac{\sfD^\mathrm{f}(C^*(X))}{\sfD^\mathrm{c}(C^*(X))}\simeq 
\frac{\sfD(C_*(\Omega X)\finbuiltby C_*(\Omega F))}{\Thick (\Q)}=\Dcosg(C_*(\Omega X)).$$
%$$\sfD^b(\mathrm{Proj}(C^*(X)))=
$$\Dcosg (C^*(X)) =\frac{\sfD^\mathrm{f}(C^*(X))}{\Thick(\Q)}%{\sfD^f_{tors}(C^*(X))}
\simeq 
\frac{\sfD(C_*(\Omega X)\finbuiltby C_*(\Omega F)))}{
\sfD^\mathrm{c}(C_*(\Omega X))}=\Dsg(C_*(\Omega X)).$$
\end{ex}

\begin{ex} {\em (Representation theory)}
\label{ex:stmodkG}
We may take $R=C^*(BG)$ for $G $ a $p$-group, since we have observed
this is g-regular. We note that $\cE\simeq kG$ and $\Dsg (kG)=\sfD^\mathrm{b}
(kG)/\sfD^\mathrm{c}(kG)$ is the stable module category, so our theorem shows
$$\Dcosg (C^*(BG))\simeq \mathrm{stmod}(kG). $$

It may be worth displaying here the correpondences amongst categories
of $C^*(BG)$-modules and $kG$-modules. Here $\sfD_{tors} (C^*(BG))$ denotes the full subcategory consisting of $C^*(BG)$-modules with homotopy that is torsion with respect to the ideal $H^{>0}(BG)$. Our equivalence of bounded
derived categories is the final row, whereas the top equivalence is
proved in \cite[7.4]{kappa} and the middle equivalence follows easily.
$$\diagram 
\Loc_{C^*(BG)} (k) \ar@{=}[r] 
&\sfD_{tors} (C^*(BG)) \ar@{<->}[r]^-{\simeq} 
&\sfD (kG) \ar@{=}[r] &\Loc_{kG} (kG)\\ 
\Thick_{C^*(BG)} (k) \ar@{=}[r] &\sfD^\mathrm{b}_{tors} (C^*(BG)) \ar@{^{(}->}[u]\ar@{^{(}->}[d]
\ar@{<->}[r]^-{\simeq} &\sfD^\mathrm{c} (kG) \ar@{^{(}->}[u]\ar@{^{(}->}[d]\ar@{=}[r] &\Thick_{kG} (kG)\\
\Thick_{C^*(BG)}(C^*(BG)) \ar@{=}[r] &\sfD^\mathrm{b} (C^*(BG)) \ar@{<->}[r]^-{\simeq} &\sfD^\mathrm{b} (kG) \ar@{=}[r]  &\Thick_{kG} (k)
\enddiagram$$
To see this makes sense, note  that since $G$ is a $p$-group $k\finbuilds kG$ and
$C^*(BG)\finbuilds k$. In particular,
$\sfD^\mathrm{b}(C^*(BG))=\sfD^\mathrm{c}(C^*(BG))$. 

\end{ex}

We next consider more general finite groups; in this case $C^*(BG)$ is generally not g-regular.

\begin{ex}
We may consider $R=C^*(BG)$ even if $G$ is not a $p$-group, and this
gives a large class of examples which formed a major motivation for
our work. In this case we may use the  normalization
arising from a faithful representation $G\lra U(n)$ (cf.\ Example~\ref{ex:BGnormalization}). Since
$H^*(BU(n))$ is polynomial, we see from Lemma~\ref{lem:creg} that a $C^*(BG)$-module is finitely generated if and
only if $H^*(M)$ is finitely generated over $H^*(BG)$. As in Example~\ref{ex:rat11} we denote the full subcategory of such modules by $\sfD^\mathrm{f}(C^*(BG))$.

However the ring $\cE \simeq C_*(\Omega (BG_p^{\wedge}))$ (see Example~\ref{ex:g}) is usually not
finite dimensional. In any case the counterpart of the previous
example is
$$\Dcosg(C^*(BG))=\frac{\sfD^\mathrm{f}(C^*(BG))}{\sfD^\mathrm{f}_{tors}(C^*(BG))}\simeq
\Dsg(\cE), $$
where $\sfD^\mathrm{f}_\mathrm{tors}(C^*(BG))$ denotes the full subcategory consisting of modules with finitely generated torsion homology. The right hand side may perhaps deserve the name
$\mathrm{stmod}(\cE)$. 

Now that  $C^*(BG)$ is usually not g-regular, the equivalence 
$$\Dsg(C^*(BG))=\frac{\sfD^\mathrm{f}(C^*(BG))}{\sfD^\mathrm{c}(C^*(BG))}\simeq
\Dcosg(\cE)$$
is also of potential interest. 
\end{ex}

\begin{ex}
We could look at the very simple example of a finite cyclic group $C$
of order $n$. Embedding $C$ in the circle group $\T$ we obtain a
fibration 
$$B\T \lla BC \lla \T /C \lla \T \lla C \lla \Omega (\T/C)\lla \Omega
\T. $$
If we suppose $C$ is a $p$-group, i.e.\ a cyclic group of prime power
order, this is also a $p$-adic fibration (i.e., a fibration after
$p$-adic completion). 

Thus, taking $k$ of characteristic $p$ and $S=C^*(B\T)$
as normalization of $R=C^*(BC)$ we find $Q=C^*(\T
/C)$ and 
$$\mcF \lla \mcE \lla \mcD$$
is 
$$C_*(\T) \lla C_*(C) \lla C_*(\Omega \T /C), $$
or algebraically
$$\Lambda [\tau]\lla k[t,t^{-1}]/(t^n-1)\lla k[t,t^{-1}]. $$
We thus see the singularity and cosingularity categories are
completely algebraic:
$$\sfD_\mathrm{sg}(C^*(BC))\simeq \sfD_\mathrm{cosg} (k[t,t^{-1}]/(t^n-1))
=\frac{\sfD^\mathrm{b} (\mbox{$k[t,t^{-1}]/(t^n-1)$})}{\thick(k)}$$
and 
$$\sfD_\mathrm{cosg}(C^*(BC))\simeq\sfD_\mathrm{sg} (k[t,t^{-1}]/(t^n-1))
=\frac{\sfD^\mathrm{b} (\mbox{$k[t,t^{-1}]/(t^n-1)$})}
{\sfD^\mathrm{c} (\mbox{$k[t,t^{-1}]/(t^n-1)$})}. $$
Since $k[t,t^{-1}]/(t^n-1)$ is a finite dimensional algebra, it is
coregular, and the first of these is trivial. However it is not
regular, so the second is not. 
\end{ex}

\begin{ex}
As  a more complicated variant, we pick an odd prime $p$ and suppose $q$ is  such that 
$q|(p-1)$ ($q$ need not be prime). We may then form the semi-direct product $G=C_p\sdr C_q$
and take $k=\Fp$. In this case a generator of $C_q$ acts on
$H^1(BC_p)=\Hom(C_p, k)$ (and hence also on $H^2(BC_p)$) as multiplication by a primitive $q$th root of $1$. Thus
$$H^*(BG)=H^*(BC_p)^{C_q}=\left[ k[x_2]\tensor
  \Lambda_k(\tau_1)\right]^{C_q}=k[X_{2q}]\tensor \Lambda_k(T_{2q-1})$$
where $X=x^q$ and $T=x^{q-1}\tau$. 

If $q=2$ then $G=D_{2p}$ is a dihedral group and has a faithful
representation $\rho$ in $U(2)$. This does not map into $SU(2)$, but 
if we complete at $p$ then the map 
$$BG\lra BU(2)\stackrel{Bdet}\lra BU(1)$$
is null since $BG$ is $p$-adically $(2q-2)$-connected, and hence we
obtain a map  $BG\lra BSU(2)$. Here the second Chern class $c_2$ maps
non-trivially since $H^*(BG)$ is finite over $H^*(BU(2))$ and hence we
have a $p$-adic  fibration 
$$S^3\lra BD_{2p}\lra BSU(2). $$

More generally we start with the natural map $BC_p\lra BU(1)$ and take
homotopy $C_q$ fixed points to obtian 
$$BG=(BC_p)^{hC_q}\lra BU(1)^{hC_q}=BS^{2q-1}$$
where $S^{2q-1}$ is the $p$-adic sphere considered as an $H$-space. 
In cohomology this is 
$$k[X]\tensor \Lambda_k(T)\lla k[X]$$
so we have a $p$-adic fibration 
$$ S^{2q-1}\lra BG\lra BS^{2q-1}. $$
Taking cochains we obtain 
$$Q\lla R\lla S, $$
and notice it satisfies the hypotheses of Proposition \ref{prop:fullsymmetric} to get a
Symmetric Gorenstein Context. 

The Eilenberg-Moore theorem shows immediately that 
\begin{align*}
H^*(\mcD)&=H_*(\Omega S^{2q-1})=\Fp [Y_{2q-2}], \\
H^*(\cE)&=H_*(\Omega (BG_p^{\wedge}))=\Fp [Y_{2q-2}]\tensor \Lambda
(U_{2q-1}), \\
 H^*(\cF)&=H_*(S^{2q-1})= \Lambda (U_{2q-1}).
\end{align*}
In particular both $R$ and $\cE$ have polynomial normalizations, so 
that finitely generated modules are those whose homology is finitely
generated over the coefficients. We learn from Theorem~\ref{thm:main} that
$$\sfD_\mathrm{sg}(C^*(BG))\simeq
\sfD_\mathrm{cosg} (C_*(\Omega (BG_p^{\wedge})))$$
and 
$$\sfD_\mathrm{cosg}(C^*(BG))\simeq
\sfD_\mathrm{sg} (C_*(\Omega (BG_p^{\wedge}))). $$
We will describe the actual category elsewhere. 
\end{ex}

The above examples all have periodic cohomology. We turn to a  related rank 2
example.  

\begin{ex}
\label{eg:DsgA4}
We take the faithful
representation of $A_4$ in $SO(3)$, and note that it gives a 2-adic
fibration
$$BSO(3)\lla BA_4\lla S^3$$
(the notable thing is Poincar\'e's result that the fibre is a 2-adic
sphere; for this and more details of the calculation, see
\cite[Example 13.3]{BGS}). 
Taking cochains to get 
$$S\lra R\lra Q$$
this corresponds to a hypersurface. 

The Eilenberg-Moore spectral sequence converges, so 
$$\mcF\lla \mcE\lla \mcD$$ 
is obtained by taking chains of 
$$SO(3)\lla X \lla \Omega S^3$$
where 
$$X= \Omega ((BA_4)_2^{\wedge}). $$
We have
$$H_*(SO(3))=\Lambda (\tau_1, \tau_2), $$
$$H_*(\Omega S^3)=k[x_2]$$
and 
$$H_*(X)=\Lambda (\sigma_1)\tensor k\langle \alpha_2,
\beta_2\rangle/(\alpha^2, \beta^2) . $$
We see that the spectral sequence of the fibration collapses and so
the map 
$$H_*(\Omega S^3 )\lra H_*(X)$$
is non-trivial and by symmetry $x$ maps to $\alpha+\beta$. 

We conclude
$$\sfD_\mathrm{sg} (C^*(BA_4))\simeq \frac{\sfD (\mbox{$C_*(X)$}
  \finbuiltby_{i^*} C_*(\Omega S^3))}{\sfD^\mathrm{f}(C_*(X))}. $$
\end{ex}

\begin{ex} 
\label{eg:Dsgchrom}
\newcommand{\Ftwo}{\mathbb{F}_2}
\newcommand{\Z}{\mathbb{Z}}
\newcommand{\cA}{\mathcal{A}}
\newcommand{\cB}{\mathcal{B}}
\newcommand{\bbS}{\mathbb{S}}
\newcommand{\ko}{\mathbf{ko}}
\newcommand{\ku}{\mathbf{ku}}
\newcommand{\tmfot}{\mathbf{tmf_1(3)}}
\newcommand{\tmf}{\mathbf{tmf}}
\newcommand{\sm}{\wedge}
There is another family of examples along the lines of Example
\ref{ex:stmodkG}, which give a partial answer to  a question of
A.J.Baker  (private communication).   This involves certain important
objects of homotopy theory. We  will not attempt a full introduction,
but give references to where the reader can find further background. 
From the point of view of this paper, these are just connective commutative
ring spectra whose homology and cohomology are as described below in
terms of subalgebras of the mod 2 Steenrod algebra $\cA$.

We take  $k=\Ftwo$ and work in a 2-complete setting so that $\Z$
denotes the 2-adic integers and $R$ is  the 2-completion of one of the ring spectra
\begin{itemize}
\item $\ku$, connective complex $K$-theory, with coefficients
  $\ku_*=\Z [v]$ where $v$ is the Bott element of degree 2
\item $\ko$, connective real $K$-theory 
\item $\tmfot$  topological modular forms  with level  structure (also
  known as $BP\langle 2 \rangle$). This has coefficients $\tmfot_*=\Z
  [v_1,v_2]$ where $v_1$ is of degree 2 and $v_2$ is of degree 6. 
\item $\tmf$, topological modular forms
\end{itemize}
Beyond the coeffients of $\ku$ and $\tmfot$ we will use two significant
facts.
\begin{itemize}
\item There is a ring map $\ko\lra \ku$ and there is an equivalence of 
  $\ko$-modules $\ku \simeq \ko \sm (S^0\cup_{\eta}e^2). $ (Connective
  version of Wood's Theorem, \cite[Lemma 4.1.2]{kobg}). 
\item There is a ring map $\tmf \lra \tmfot$ and there is an equivalence of 
  $\tmf$-modules $\tmfot \simeq \tmf \sm D\cA (1). $ where $D\cA (1)$
  is a self-dual 8-cell complex with cells in dimensions $0, 2, 4, 6, 6,
  8, 10, 12$. (Hopkins-Mahowald \cite{Mathewtmf})
\end{itemize}

Using these facts we see all four rings $R$ are regular. This is obvious for $\ku$ and
$\tmfot$, from their coefficients. For $\ko$ we use regularity of
$\ku$ and Wood's Theorem. For $\tmf$ we use regularity of $\tmfot$
and the Hopkins-Mahowald theorem. 
All four spectra $R$ are bounded below and each homotopy group is
finitely generated as a $\ZZ$-module, and
hence each spectrum $R$ has a locally finite mod 2 Adams resolution.  We deduce that $R$ and $\cE$ are 
equal to their double centralizers in the sense of \cite[4.16]{DGI}.  

The interest for us comes from the fact that, $\cE =\Hom_R(\Ftwo, \Ftwo)$ has homotopy given by the
appropriate finite dimensional Hopf subalgebra  $\cB$ of the mod 2 Steenrod algebra $\cA$, namely 
\begin{multline*}
\cE (1)=\Lambda (Q_0, Q_1),\;  \cA (1)=\langle Sq^1, Sq^2\rangle,\;\\
\cE (2)=\Lambda (Q_0, Q_1, Q_2),\; 
\cA (2)=\langle Sq^1, Sq^2, Sq^4\rangle
\end{multline*}
respectively. We will explain how to deduce this from well-known calculations.

In each case the mod 2 cohomology 
$k^*(R)=\Hom_{\bbS}(R,k)_*$ is known to be a quotient algebra of
$k^*(k)=\Hom_{\bbS}(k,k)_*$, since $k^*R=\cA\tensor_{\cB}k$:
\begin{itemize}
\item $H^*(ku; \Ftwo)=\cA \tensor_{\cE (1)}\Ftwo$ \cite{AdamsChern}
\item $H^*(ko;\Ftwo)=\cA \tensor_{\cA (1)}\Ftwo$ \cite[Part III, 6.6]{AdamsBlue}
\item $H^*(BP\langle 2\rangle; \Ftwo)=\cA \tensor_{\cE (2)}\Ftwo$ \cite[Proposition 1.7]{WilsonII} 
\item $H^*(tmf; \Ftwo)=\cA \tensor_{\cA (2)}\Ftwo$ \cite{HopkinsMahowald, Mathewtmf}
\end{itemize}
We  use these calculations  to deduce that $\cE_*=\cB$. 
First note that $\Hom_R(k, k)\simeq \Hom_{k\tensor R} (k \tensor k, k)$, and 
hence there is a spectral sequence 
$$\Ext_{k_*R}^{*,*}( k_*k, k)\Rightarrow \pi_*(\Hom_R(k, k)). $$
Since $k_*R\lra k_*k$ is injective by the quoted calculations, and a 
map of commutative Hopf algebras,  $k_*k=\cA_*$ is free over $k_*R$ and the 
spectral sequence collapses to show 
$$\cE_*=\pi_*(\Hom_R(k,k))=\Hom_{k_*R}(\cA_*,k)=\cB.  $$
%Note this also shows  the map to $\Hom_k(\cA_*, k)=\cA^*$ is injective. 

The following more general structural statement helps make sense of
this. 
\begin{lem}
\label{lem:cof1}
If $R\lra k$ is proxy-regular and $R$ is bounded below then there is a cofibre sequence 
%\begin{equation}
%\label{cof1}
$$\Hom_{R}(k, k)\lra 
\Hom_{\bbS}(k, k)\lra 
\Hom_{\bbS}(R, k) 
$$
%\end{equation}
of algebras augmented over $k$. 
\end{lem}
\begin{proof}
First we note
$\Hom_{\bbS}(k,k)=\Hom_R(k, \Hom_{\bbS}(R,k))$. Next, we note $R\lra k$ is proxy 
regular and hence by \cite[6.10]{DGI} there is an equivalence
$$\Hom_R(k, \Hom_{\bbS}(R,k))\tensor_{\Hom_R(k,k)}k\simeq \mathrm{Cell}_k 
\Hom_{\bbS}(R,k)\simeq \Hom_{\bbS}(R,k),  $$
where the $k$-cellularity of $\Hom_{\bbS}(R,k)$ comes from the fact 
that it is locally finite and bounded above.
\end{proof}

%We now argue that $\cE_*\cong \cB$. 

%Accordingly, since $\cE_*=\cB$ 

\begin{rem}  We see that in our case the homotopy of the cofibre sequence in Lemma
\ref{lem:cof1} is the multiplicative short exact
sequence
$$\cB \lra \cA \lra \cA\tensor_{\cB}k. $$
It would be nice to reverse the argument we have given:
% Ideally we would argue briefly as follows
the map $\Hom_{\bbS}(k,k)\lra \Hom_{\bbS}(R,k)$ is surjective in
homotopy since $k^*R=\cA\tensor_{\cB}k$, and hence the spectral
sequence of the cofibre sequence in Lemma \ref{lem:cof1} collapses.  Hence $\Hom_R(k,k)\lra \Hom_{\bbS}(k,k)$
is injective and $\cE_*\cong \cB$ as required. However the relevant properties of the spectral
sequence are not documented.
\end{rem} 
%, so we give a more direct and explicit
%%argument shortly. 
%For readers interested in the singularity category, 

We may now observe that in this case, Theorem \ref{thm:main} states 
$$\Dcosg (R)\simeq \Dsg (\cE), $$
where $\Dsg (\cE)$ can be viewed as a lifting of
$\mathrm{stmod}(\cE_*). $\\[2ex]

\end{ex}

\begin{rem}
It seems to be an interesting problem to give criteria weaker than formality for an
equivalence $\sfD_\mathrm{sg}(A)\simeq \sfD_\mathrm{sg}(H_*(A))$. This
is probably fairly rare. For example  if $A=C^*(BG)$ for a $p$-group $G$ then
$\sfD_\mathrm{sg}(A)\simeq 0$ but the cohomology ring $H^*(BG)$ is
usually not  regular so $\sfD_\mathrm{sg}(H_*(A))\not \simeq 0$ (the
smallest examples are the dihedral and quaternion groups of order 8). 
\end{rem}

\section{Glossary}\label{sec:glossary}
This section contains a sorted list of key terminology, together with references to the appropriate definitions within the text. Entries appear, under each category, in the order they are defined in the text.

\subsection{Properties of ring spectra and maps}
When they are defined in the same place the relative version of a property is not listed separately.

\begin{itemize}

\item{\emph{g-regular:}} Section~\ref{sec:regularity}

\item{\emph{proxy-regular}:} Definition~\ref{defn:proxysmall}

\item{\emph{normalization}:} Definition~\ref{defn:g-normalization}

\item{\emph{coefficient-regular}:} Definition~\ref{defn:creg}

\item{\emph{coefficient-normalization}:} Section~\ref{ssec:crn}

\item{\emph{Gorenstein}:} Definition~\ref{defn:gor}

\item{\emph{strongly Gorenstein normalization}:} Definition~\ref{defn:sgn}

\item{\emph{complete}:} Section~\ref{ssec:completion}

\item{\emph{coregular}:} Definition~\ref{defn:coregular}

\end{itemize}

\subsection{Finiteness properties for modules}

\begin{itemize}

\item{\emph{coefficient-finitely generated}:} Definition~\ref{defn:c-fg}

\item{\emph{q-finitely generated}:} Definition~\ref{defn:q-fg}

\item{\emph{locally finitely generated/presented}:} Definition~\ref{def:lfp}

\item{\emph{cohomologically locally finitely generated/presented}:} Definition~\ref{defn:clfp}

%\item{\emph{}:}

\end{itemize}

\subsection{The fundamental setup and categories}

\begin{itemize}

\item{\emph{Symmetric Gorenstein Context}:} Definition~\ref{defn:sgc}

\item{$\sfD^{q\mathrm{-b}}(R)$:} Definition~\ref{defn:q-fg}

\item{$\sfD_{q\mathrm{-sg}}(R)$:} Definition~\ref{defn:qsg}

\item{$\sfD_{q\mathrm{-cosg}}(R)$:} Definition~\ref{defn:qcsg}

\end{itemize}

\bibliography{greg_bib}

\begin{bibdiv}
\begin{biblist}

\bib{AdamsChern}{article}{
      author={Adams, J.~F.},
       title={On {C}hern characters and the structure of the unitary group},
        date={1961},
        ISSN={0008-1981},
     journal={Proc. Cambridge Philos. Soc.},
      volume={57},
       pages={189\ndash 199},
  url={https://0-doi-org.pugwash.lib.warwick.ac.uk/10.1017/s0305004100035052},
      review={\MR{121795}},
}

\bib{AdamsBlue}{book}{
      author={Adams, J.~F.},
       title={Stable homotopy and generalised homology},
   publisher={University of Chicago Press, Chicago, Ill.-London},
        date={1974},
        note={Chicago Lectures in Mathematics},
      review={\MR{0402720}},
}

\bib{Amiot}{article}{
      author={Amiot, Claire},
       title={Cluster categories for algebras of global dimension 2 and quivers
  with potential},
        date={2009},
     journal={Ann. Inst. Fourier (Grenoble)},
      volume={59},
      number={6},
       pages={2525\ndash 2590},
}

\bib{AFHdescent}{article}{
      author={Avramov, Luchezar~L.},
      author={Foxby, Hans-Bj{\o}rn},
      author={Halperin, Stephen},
       title={Descent and ascent of local properties along homomorphisms of
  finite flat dimension},
        date={1985},
        ISSN={0022-4049},
     journal={Journal of Pure and Applied Algebra},
      volume={38},
      number={2},
       pages={167 \ndash  185},
  url={http://www.sciencedirect.com/science/article/pii/0022404985900076},
}

\bib{BG7}{article}{
      author={Benson, David},
      author={Greenlees, J.P.C.},
       title={Stratifying the derived category of cochains on {BG} for {G} a
  compact {L}ie group},
        date={2014},
     journal={Journal of Pure and Applied Algebra},
      volume={218},
      number={4},
       pages={642\ndash 650},
}

\bib{kappa}{article}{
      author={Benson, David~J.},
      author={Greenlees, J. P.~C.},
       title={Localization and duality in topology and modular representation
  theory},
        date={2008},
        ISSN={0022-4049},
     journal={J. Pure Appl. Algebra},
      volume={212},
      number={7},
       pages={1716\ndash 1743},
  url={https://0-doi-org.pugwash.lib.warwick.ac.uk/10.1016/j.jpaa.2007.12.001},
      review={\MR{2400738}},
}

\bib{BGS}{article}{
      author={Benson, David~J.},
      author={Greenlees, J. P.~C.},
      author={Shamir, Shoham},
       title={Complete intersections and mod {$p$} cochains},
        date={2013},
        ISSN={1472-2747},
     journal={Algebr. Geom. Topol.},
      volume={13},
      number={1},
       pages={61\ndash 114},
  url={https://0-doi-org.pugwash.lib.warwick.ac.uk/10.2140/agt.2013.13.61},
      review={\MR{3031637}},
}

\bib{kobg}{book}{
      author={Bruner, Robert~R.},
      author={Greenlees, J. P.~C.},
       title={Connective real {$K$}-theory of finite groups},
      series={Mathematical Surveys and Monographs},
   publisher={American Mathematical Society, Providence, RI},
        date={2010},
      volume={169},
        ISBN={978-0-8218-5189-0},
         url={https://0-doi-org.pugwash.lib.warwick.ac.uk/10.1090/surv/169},
      review={\MR{2723113}},
}

\bib{Cohen46}{article}{
      author={Cohen, Irvin~S},
       title={On the structure and ideal theory of complete local rings},
        date={1946},
     journal={Transactions of the American mathematical Society},
      volume={59},
      number={1},
       pages={54\ndash 106},
}

\bib{DGI}{article}{
      author={Dwyer, W.},
      author={Greenlees, J. P.~C.},
      author={Iyengar, S.},
       title={Finiteness in derived categories of local rings},
        date={2006},
     journal={Comment. Math. Helv.},
      volume={81},
      number={2},
       pages={383\ndash 432},
}

\bib{DG}{article}{
      author={Dwyer, W.~G.},
      author={Greenlees, J. P.~C.},
       title={Complete modules and torsion modules},
        date={2002},
     journal={Amer. J. Math.},
      volume={124},
      number={1},
       pages={199\ndash 220},
}

\bib{DGI2}{article}{
      author={Dwyer, W.~G.},
      author={Greenlees, J. P.~C.},
      author={Iyengar, S.},
       title={Duality in algebra and topology},
        date={2006},
     journal={Adv. Math.},
      volume={200},
      number={2},
       pages={357\ndash 402},
}

\bib{DGI4}{article}{
      author={Dwyer, W.~G.},
      author={Greenlees, J. P.~C.},
      author={Iyengar, S.~B.},
       title={D{G} algebras with exterior homology},
        date={2013},
     journal={Bull. Lond. Math. Soc.},
      volume={45},
      number={6},
       pages={1235\ndash 1245},
}

\bib{elagin2018smoothness}{article}{
      author={Elagin, Alexey},
      author={Lunts, Valery~A},
      author={Schn{\"u}rer, Olaf~M},
       title={Smoothness of derived categories of algebras},
        date={2018},
     journal={arXiv preprint arXiv:1810.07626},
}

\bib{FHT}{article}{
      author={F\'{e}lix, Yves},
      author={Halperin, Stephen},
      author={Thomas, Jean-Claude},
       title={Gorenstein spaces},
        date={1988},
        ISSN={0001-8708},
     journal={Adv. in Math.},
      volume={71},
      number={1},
       pages={92\ndash 112},
  url={https://0-doi-org.pugwash.lib.warwick.ac.uk/10.1016/0001-8708(88)90067-9},
      review={\MR{960364}},
}

\bib{GM}{article}{
      author={Greenlees, J. P.~C.},
      author={May, J.~P.},
       title={Derived functors of {$I$}-adic completion and local homology},
        date={1992},
     journal={J. Algebra},
      volume={149},
      number={2},
       pages={438\ndash 453},
}

\bib{JohnNotes}{article}{
      author={Greenlees, J.P.C.},
       title={Homotopy invariant commutative algebra over fields},
        date={2016},
     journal={arXiv:1601.02473},
}

\bib{bgen}{article}{
      author={Greenlees, J.P.C.},
       title={Borel cohomology and the relative {G}orenstein condition for
  classifying spaces of compact {L}ie groups},
        date={2018},
     journal={arXiv:1808.07342},
}

\bib{HopkinsMahowald}{incollection}{
      author={Hopkins, Michael~J.},
      author={Mahowald, Mark},
       title={From elliptic curves to homotopy theory},
        date={2014},
   booktitle={Topological modular forms},
      series={Math. Surveys Monogr.},
      volume={201},
   publisher={Amer. Math. Soc., Providence, RI},
       pages={261\ndash 285},
         url={https://0-doi-org.pugwash.lib.warwick.ac.uk/10.1090/surv/201/15},
      review={\MR{3328536}},
}

\bib{KellerDCY}{article}{
      author={Keller, Bernhard},
       title={Deformed {C}alabi-{Y}au completions (with an appendix by {M}ichel
  {V}an den {B}ergh)},
        date={2011},
     journal={J. Reine Angew. Math.},
      volume={654},
       pages={125\ndash 180},
}

\bib{Lunts}{article}{
      author={Lunts, Valery~A.},
       title={Categorical resolution of singularities},
        date={2010},
     journal={J. Algebra},
      volume={323},
      number={10},
       pages={2977\ndash 3003},
}

\bib{Mathewtmf}{article}{
      author={Mathew, Akhil},
       title={The homology of tmf},
        date={2016},
        ISSN={1532-0073},
     journal={Homology Homotopy Appl.},
      volume={18},
      number={2},
       pages={1\ndash 29},
  url={https://0-doi-org.pugwash.lib.warwick.ac.uk/10.4310/HHA.2016.v18.n2.a1},
      review={\MR{3515195}},
}

\bib{McCleary}{book}{
      author={McCleary, John},
       title={A user's guide to spectral sequences},
     edition={Second},
      series={Cambridge Studies in Advanced Mathematics},
   publisher={Cambridge University Press, Cambridge},
        date={2001},
      volume={58},
        ISBN={0-521-56759-9},
      review={\MR{1793722}},
}

\bib{NeemanLoc}{article}{
      author={Neeman, Amnon},
       title={The connection between the {K}-theory localization theorem of
  {T}homason, {T}robaugh and {Y}ao and the smashing subcategories of
  {B}ousfield and {R}avenel},
    language={eng},
        date={1992},
     journal={Ann. Sci. \'Ec. Norm. Sup\'er.},
      volume={25},
      number={5},
       pages={547\ndash 566},
         url={http://eudml.org/doc/82328},
}

\bib{NeeGrot}{article}{
      author={Neeman, Amnon},
       title={The {G}rothendieck duality theorem via {B}ousfield's techniques
  and {B}rown representability},
        date={1996},
     journal={J. Amer. Math. Soc.},
      volume={9},
      number={1},
       pages={205\ndash 236},
}

\bib{neeman2017strong}{article}{
      author={Neeman, Amnon},
       title={Strong generators in ${D}^{perf}({X})$ and
  ${D}^b(\mathrm{coh}({X}))$},
        date={2017},
     journal={arXiv preprint arXiv:1703.04484},
}

\bib{neemanTBC}{article}{
      author={Neeman, Amnon},
       title={The categories $\mathcal{T}^c$ and $\mathcal{T}^b_c$ determine
  each other},
        date={2018},
     journal={arXiv preprint arXiv:1806.06471},
}

\bib{OS12}{article}{
      author={Oppermann, Steffen},
      author={\v{S}t'ov\'{\i}\v{c}ek, Jan},
       title={Generating the bounded derived category and perfect ghosts},
        date={2012},
     journal={Bull. Lond. Math. Soc.},
      volume={44},
      number={2},
       pages={285\ndash 298},
}

\bib{Quillen}{article}{
      author={Quillen, Daniel},
       title={Rational homotopy theory},
        date={1969},
        ISSN={0003-486X},
     journal={Ann. of Math. (2)},
      volume={90},
       pages={205\ndash 295},
      review={\MR{0258031}},
}

\bib{Rdim}{article}{
      author={Rouquier, Rapha{\"e}l},
       title={Dimensions of triangulated categories},
        date={2008},
     journal={J. K-Theory},
      volume={1},
      number={2},
       pages={193\ndash 256},
}

\bib{SchwedeShipleymodule}{article}{
      author={Schwede, Stefan},
      author={Shipley, Brooke},
       title={Stable model categories are categories of modules},
        date={2003},
     journal={Topology},
      volume={42},
      number={1},
       pages={103\ndash 153},
}

\bib{ShipleyHZ}{article}{
      author={Shipley, Brooke},
       title={{HZ}-algebra spectra are differential graded algebras},
        date={2007},
        ISSN={00029327, 10806377},
     journal={American Journal of Mathematics},
      volume={129},
      number={2},
       pages={351\ndash 379},
}

\bib{WilsonII}{article}{
      author={Wilson, W.~Stephen},
       title={The {$\Omega $}-spectrum for {B}rown-{P}eterson cohomology.
  {II}},
        date={1975},
        ISSN={0002-9327},
     journal={Amer. J. Math.},
      volume={97},
       pages={101\ndash 123},
         url={https://0-doi-org.pugwash.lib.warwick.ac.uk/10.2307/2373662},
      review={\MR{383390}},
}

\end{biblist}
\end{bibdiv}

\end{document}